 \documentclass[1p]{elsarticle}
 
 \usepackage[latin1]{inputenc}

\journal{Journal of \LaTeX\ Templates}

\usepackage{amssymb}
\usepackage{eucal}
\newcommand{\R}{{\Bbb R}}

\newcommand{\N}{{\Bbb N}}

\usepackage{color}

\usepackage{xcolor}
\definecolor{alizarin}{rgb}{0.82, 0.1, 0.26}
\definecolor{burgundy}{rgb}{0.5, 0.0, 0.13}
\definecolor{darkblue}{rgb}{0.0, 0.0, 0.55}
\definecolor{deepcarmine}{rgb}{0.66, 0.13, 0.24}
\definecolor{navyblue}{rgb}{0.0, 0.0, 0.5}

\usepackage{amsmath}

\newtheorem{thm}{Theorem}
\newtheorem{lem}[thm]{Lemma}
\newtheorem{cor}[thm]{Corollary}
\newtheorem{proposition}[thm]{Proposition}

\newtheorem{example}[thm]{Example}
\newproof{proof}{Proof}
\begin{document}
\begin{frontmatter}

\title{Two reasons for the appearance of  pushed wavefronts in the Belousov-Zhabotinsky system with spatiotemporal interaction}

\author[b]{Karel Has\'ik}
\author[b]{Jana Kopfov\'a}
\author[b]{Petra N\'ab\v{e}lkov\'a}
\author[c]{Olena Trofymchuk}
\author[d]{Sergei Trofimchuk\corref{mycorrespondingauthor}}
\cortext[mycorrespondingauthor]{\hspace{-7mm} {\it e-mails addresses}:  karel.hasik@math.slu.cz (Karel Has\'ik);   jana.kopfova@math.slu.cz (Jana Kopfov\'a); petra.nabelkova@math.slu.cz (Petra N\'ab\v{e}lkov\'a); o.trofimchyk@kpi.ua (Olena Trofymchuk); trofimch@inst-mat.utalca.cl (Sergei Trofimchuk) \\}
\address[b]{Mathematical Institute, Silesian University, 746 01 Opava, Czech Republic}
\address[c]{Department of Mathematical Physics and Differential Equations, Igor Sikorsky Kyiv Polytechnic Institute, Kyiv, Ukraine}
\address[d]{Instituto de Matem\'atica y F\'isica, Universidad de Talca, Casilla 747,
Talca, Chile }

\bigskip

\begin{abstract}
\noindent  We prove the existence minimal speed of propagation $c_*(r,b,K) \in [2\sqrt{1-r},2]$ for wavefronts in the Belousov-Zhabotinsky system with a spatiotemporal interaction defined by the convolution with (possibly,  "fat-tailed") kernel $K$.   The model is assumed to be monostable non-degenerate, i.e. $r\in (0,1)$.  The slowest wavefront  is termed pushed or non-linearly determined  if its velocity  $c_*(r,b,K) > 2\sqrt{1-r}$.  We show that $c_*(r,b,K)$ is close to 2 if i) positive system's parameter $b$ is sufficiently large or ii) if $K$ is spatially asymmetric to one side (e.g. to the left: in such a case, the influence of the right side concentration of the bromide ion on the  dynamics  is more significant than the influence of the left  side).  Consequently, this reveals two reasons for the appearance of pushed wavefronts in the  Belousov-Zhabotinsky reaction.  
 \end{abstract}
\begin{keyword} nonlocal delay, wavefront,  reaction-diffusion, Belousov-Zhabotinsky reaction  \\
{\it 2010 Mathematics Subject Classification}: {\ 34K12, 35K57,
92D25 }
\end{keyword}

\end{frontmatter}

\newpage

\section{Nonlinearly determined wavefronts in the Belousov-Zhabotinsky system}
\subsection{Introduction}  \label{intro} 
In this paper, we consider the monostable reaction-diffusion system 
\begin{equation}\label{1}
\begin{array}{ll}
     u_t(t,x) = \Delta u(t,x)  + u(t,x)(1-u(t,x)-r(K*v)(t,x)),
    &    \\
     v_t(t,x) = \Delta v(t,x)  -b u(t,x)v(t,x), \ u, v \geq 0, \ x \in \R^n,& 
\end{array}%
\end{equation}
following J.D. Murray  studies \cite{Mur1, Mur2} of traveling waves in the Noyes-Field theory of the Belousov-Zhabotinsky (BZ for short)  chemical reaction.  The variables $u, v$ represent the bromous acid and bromide ion concentrations, respectively.  We will assume that the system  (\ref{1}) is monostable non-degenerate that amounts to the condition $r \in (0,1)$, 
cf. \cite{TPTa}. The real parameter $b$ is positive (by  \cite{Troy}, $b \approx 20$  for a real chemical experiment). 
$K*v$ denotes the convolution of the component $v$ with the non-negative normalised kernel $K(s,y), \ s \geq 0, y \in \R$: 
$$
\int_0^{+\infty}\int_\R K(s,y)\;dy\;ds=1.
$$

The wavefront 
$(u,v) = (\phi, \theta)(\nu\cdot x + ct),$  $\|\nu\| =1,$ is a positive $C^2$-smooth solution of  (\ref{1})  satisfying the boundary conditions 
$$
(\phi, \theta)(-\infty)=(0,1), \ (\phi, \theta)(+\infty)
= (1,0). 
$$
Equivalently, the profiles $\phi(t), \psi(t):= 1-\theta(t)$ are $C^2$-smooth solutions to the system 
\begin{equation}\label{3ade}\left\{
\begin{array}{ll}
      \phi''(t) - c\phi'(t) + \phi(t) (1 - r- \phi(t)+ r(K\star \psi)(t)) =0,
    &    \\
     \psi''(t) - c\psi'(t) +b \phi(t)(1-\psi(t)) =0, & \\      \phi>0, \psi <1, \      \phi(-\infty)=\psi(-\infty)=0, \  \phi(+\infty)=\psi(+\infty)=1,  &\\ 
\end{array}%
\right.
\end{equation}
where 
$$
(K\star\psi)(t)= \int_0^{+\infty}\int_\R K(s,y)\psi(t-cs-y)\;dy\;ds
$$
(note that $K\star\psi$ depends  on $c$, sometimes we will also write $K\star_c\psi$ indicating explicitly this dependence). 
It follows that $\phi(t), \psi(t)$ are $C^{\infty}$-smooth functions. 

After Murray's original works, system (\ref{1}) was studied by many researchers, the existence of traveling waves being one of the problems of fundamental interest.  The first analytical result concerning this problem in the local case, $(K*v)(t,x) = v(t,x)$, was obtained by Troy \cite{Troy} in 1980. Later on, in a series of papers (e.g. see \cite{bn,Ka2,ma,TPTa,Volp,wz,yw,Zh}) different approaches were developed to tackle 
the wave existence problem in the mentioned local case and also in 
the delayed case,  $(K*v)(t,x) = v(t-h,x)$.  
One of the most recent articles \cite{DQ} established the existence of waves for two special non-local kernels $K$  (see Example \ref{ws} below) by invoking  Fenichel's geometric singular perturbation theory. Yet, as we show later in this section,  complete solution to the question of the existence of wavefront  for the BZ system is not available even in the local case.  Particularly, in the present work, we determine and analyse the main factors making the described problem so difficult.  In the  next subsections, we present and comment the main results of our studies. 

%%%%%%%%%%%%%%%%%%%%%%%%%%%%%%%%%%%%%%%%%%%%%%%%%%%%%%%%%%%%%%%%%

\subsection{Semi-wavefronts and their monotonicity}
Suppose that $(\phi,\psi)$ is a bounded smooth solution  to (\ref{3ade}) satisfying  weaker  positivity and boundary conditions $$
  \phi(-\infty)=\psi(-\infty)=0, \  \phi(t) \geq 0,\  t \in \R, \quad  \liminf_{t \to +\infty}\phi(t) >0, \quad \liminf_{t \to +\infty}\psi(t) >0. 
$$
We will call such a solution $(\phi,\psi)$  a semi-wavefront, this concept is rather natural for nonlocal systems since the nonlocal interaction can affect the monotonicity of wavefronts  \cite{FZ,MF}.   It is somewhat surprising that  model (\ref{1}) is robust in this respect: 
\begin{thm} \label{mude} Assume that  $r,b >0$.  If system (\ref{3ade}) has a  semi-wavefront  $(\phi, \psi)$,  then $\phi(+\infty) = \psi(+\infty)=1$, and $\psi'(t) >0$, $\phi'(t) >0$ for all $t \in \R$.  
\end{thm}
We prove this Theorem in   Section \ref{Sec2}. \\

\noindent Let us present  two useful consequences of the above result:  

\begin{lem}  \label{mda} Let $(\phi,\psi)$ be a solution of  (\ref{3ade}) and $r \in (0,1), \ b >0$.  Then $\psi^2(t) < M\phi(t),$ $t \in \R,$  where $ M=1/(1-r)$ if $b\leq 0.5$ and $M=(1-r)^{-1}b^2/(2b-1)$ if $b>0.5$. 
\end{lem} 
\begin{proof}
 Set  $z(t)= M\phi(t) - \psi^2(t)$ and suppose that $z(t)$ is non-positive at some points.  Clearly, $z(-\infty)=0,\  z(+\infty)=M -1 \geq 0, $ so that   $z(\tau)\leq 0, \ z'(\tau)= 0, \
z''(\tau)\geq 0$ at some $\tau \in \R$. Hence,
$$
M\phi(\tau) \leq \psi^2(\tau), \quad M\phi'(\tau) = 2\psi(\tau)\psi'(\tau), \quad 
M\phi''(\tau) \geq 2\psi(\tau)\psi''(\tau)+ 2(\psi'(\tau))^2, 
$$
$$
\begin{array}{ll}
     0\geq  2\psi(\tau)\psi''(\tau)+ 2(\psi'(\tau))^2- 2c\psi(\tau)\psi'(\tau) + M\phi(\tau) (1-r +r(K\star\psi)(\tau)- \phi(\tau)),
    &    \\
     0=2\psi(\tau)\psi''(\tau) - 2c\psi(\tau)\psi'(\tau) +2b\psi(\tau)\phi(\tau)(1-\psi(\tau)),  & \\      
\end{array}%
$$
and therefore, by using Theorem \ref{mude}, we obtain 
\begin{eqnarray*}
0&\geq &  2(\psi'(\tau))^2 + M\phi(\tau) (1-r +r(K\star\psi)(\tau)- \phi(\tau)) - 
2b\psi(\tau)\phi(\tau)(1-\psi(\tau)) \\
& > &  \phi(\tau)\left\{M(1-r) -2b\psi(\tau) +\psi^2(\tau)(2b-1)\right\}\geq 0,
\end{eqnarray*}
a contradiction. 
Note  that the polynomial $p(z):= M(1-r) -2bz +z^2(2b-1), $ $z := \psi(\tau) \in (0,1), $ satisfies 
$
p(0)= M(1-r) > 0, \ p(1) = M(1-r)-1 \geq 0, 
$
so that $p(\psi(\tau)) \geq 0$   if $2b -1 \leq 0$. If $2b-1 >0$
then  $p(z) \geq \min_{z \in \R}p(z)= M(1-r)  -  b^2(2b-1)^{-1} =0$. \qed
\end{proof}

\begin{proposition} \label{mudas} Assume that  $r \in (0,1),b >0$.  If system (\ref{3ade}) has a  semi-wavefront  $(\phi, \psi)$,  then $c \geq 2\sqrt{1-r}$ and there  exist  $t_1$, $m \in \{0,1\}$  such  that   
\begin{equation*}
(\psi(t+t_1), \phi(t+t_1),\phi'(t+t_1))= 
     (-t)^me^{\nu t}(b/(1-r), 1, \nu)(1+o(1)), \  t \to -\infty,
\end{equation*}
where  $\nu$ is  one of the zeros
% \in \{\lambda(c), \mu (c)\}$ 
of the characteristic polynomial $\chi_r(z,c) = z^2-cz+1-r.$
\end{proposition}
\begin{proof} The proof of this assertion is identical to the proof of Lemma 11 and Corollary 12 in \cite{TPTa}, where we use our Lemma \ref{mda}  instead of Theorem 6 (A) in  \cite{TPTa}.  
 \qed \end{proof}
 
 %%%%%%%%%%%%%%%%%%%%%%%%%%%%%%%%%%%%%%%%%%%%%%%%%%%%%%%%%%%%

\subsection{Pulled and pushed wavefronts in the BZ system: explanation through the KPP-Fisher equation with non-local interaction}

In the special case when $b=1-r$ and  $(K\star \psi)(t)=\psi(t),$ both equations in  (\ref{3ade}) coincide, i.e. $\phi=\psi$ %\cite{TPTa} 
and  we have 
\begin{equation}\label{pp}
\phi''(t)-c\phi'(t)+(1-r)\phi(t)(1-\phi(t))=0,\ 
0< \phi <1, \      \phi(-\infty)=0, \  \phi(+\infty)=1,
\end{equation}
which is precisely the profile equation for the classical KPP-Fisher model.  
The nonlinearity $g(\phi)= (1-r)\phi(1-\phi)$ is sub-tangential at $\phi=0$ (i.e.  $g(\phi) \leq g'(0)\phi, \ \phi \geq 0$), which guarantees the existence of a monotone wavefront to (\ref{pp}) if and only if $c \geq 2\sqrt{g'(0)} = 2\sqrt{1-r}=:c_*(r,1-r)$. The minimal speed $c_*(r,1-r)$ of wavefronts in the KPP-Fisher case is completely determined by the linearisation of equation  (\ref{pp}) along  zero, the minimal wavefronts are called pulled, or linearly determined.  

Without the additional restriction  $g(\phi) \leq g'(0)\phi$, the minimal speed $c_*$ of wavefronts can be bigger than $2\sqrt{g'(0)}$, in such a case, the wavefronts propagating with the minimal speed $c_*$ are called pushed. For instance, the following modification of the KPP-Fisher profile equation 
\begin{equation}\label{ppk}
\phi''(t)-c\phi'(t)+(1-r)(1+l \phi(t))\phi(t)(1-\phi(t))=0,\ 
    \phi(-\infty)=0, \  \phi(+\infty)=1,
\end{equation}
possesses monotone solutions if and only if 
\begin{equation}\label{cs}
c \geq c_*= \sqrt{1-r}\left\{
\begin{array}{ll}
    2,
    &    l\in [0,2] ,\\
     (l+2)/\sqrt{2l}, & l \geq 2, 
     \\   \end{array}%
\right.
\end{equation}
see \cite{HR} (a simple explanation for this form of $c_*$ can be found in \cite{KH, MFa}). In the nonlinearity $g_l(\phi)= (1-r)(1+l\phi)\phi(1-\phi)$, parameter $l$  measures "the excess" of the reaction graph  $y=g_l(\phi)$ over the tangent line $y=  g_l'(0)\phi, \ \phi \geq 0$. By (\ref{cs}), if this excess is relatively small, $l \in [0,2]$, the minimal waves are still linearly determined and if  $l >2$,  the minimal wavefronts are pushed. 

Now,  as we show  in Appendix, system (\ref{3ade}) can be transformed into the following equivalent KPP-Fisher type equation 
\begin{equation}\label{KLMMM}
  \phi''(t) - c\phi'(t) + \phi(t) (1 - r- \phi(t)+ r(K\star_c (\frak{L}_b \phi))(t)) =0, \phi >0,  \phi(-\infty)=0, \ \phi(\infty)=1, 
\end{equation}
where operator $\frak{L}_b$ maps the set of non-decreasing continuous functions  $\phi >0,$ \mbox{$\phi(-\infty)=0,$}  $\phi(+\infty)=1,$  possessing finite integral   $\int_{-\infty}^0\phi(t)dt$, into the set of strictly monotone $C^2$-smooth functions $\psi(t),$ $\psi(-\infty)=0, \ \psi(+\infty)=1$. Furthermore, 
$\frak{L}_b$ commutes with the translation operator, $(\frak{L}_b\phi(\cdot+h))(t)= (\frak{L}_b\phi(\cdot))(t+h)$, and is 
monotone increasing with respect to $b$ and $\phi$.  Monotonicity guarantees  a kind of  continuity  of $\frak{L}_b$ (see Appendix) and suggests a natural extension   of $\frak{L}_b$ to the fixed points $\phi=0$ and $\phi=1$. 

Clearly, the reaction term $g(\phi)= \phi(t) (1 - r- \phi(t)+ r(K\star_c (\frak{L}_b \phi))(t))$ is not sub-tangential  at $\phi=0$ whenever $\Lambda(t):=- \phi(t)+ r(K\star_c (\frak{L}_b \phi))(t) >0, \ t \leq t_0,$ for admissible wave profiles $\phi$ and some $t_0$. 
Since $\frak{L}_b \phi=\psi$, by Proposition \ref{mudas}, each non-critical pulled wavefront satisfies 
\begin{equation*}
( (\frak{L}_b \phi)(t), \phi(t))= 
     e^{\lambda_r(c) t}(b/(1-r), 1)(1+o(1)), \  t \to -\infty,
\end{equation*}
where  $0<\lambda_r(c)$ is the smallest  zero of the characteristic polynomial 
$\chi_r(z,c)=z^2-cz +1-r$. Thus  
$$
\Lambda(t) \sim  e^{\lambda_r(c) t}\left(-1 + \frac{br \kappa(r,c)}{1-r}\right), t \to -\infty, \ \mbox{where} $$
$$
 \kappa(r,c):= \int_0^{+\infty}\int_\R K(s,y)e^{-\lambda_r(c)(y+cs)}dy\; ds, 
$$
and, consequently, the difference 
$$
\frac{br \kappa(r,c)}{1-r}-1 =\frac{r+br \kappa(r,c)-1}{1-r} 
$$
can be regarded as a measure of "the excess" of the reaction term $y=g(\phi)$ over its linear part at the steady state 
$\phi=0$. In particular, we can expect that the minimal wavefront is linearly determined when $r+br \kappa(r,c)\leq 1$, and  is nonlinearly determined when $r+br \kappa(r,c)\gg1$.  Our next results support this informal conclusion.  
\begin{thm}\label{mumiV} Assume that  positive $r, b$ and $c \geq 2\sqrt{1-r}$ satisfy $r+br\kappa(r,c) \leq 1$. Then  there exists 
at least one positive monotone wavefront for (\ref{1}) propagating at the speed $c$. 
\end{thm}

\noindent Proof of this Theorem is postponed to  Section \ref{Exi}. \qed \\

\noindent Actually, the use of the concept `minimal speed of propagation' or `minimal' or `critical' wavefront with respect to the system (\ref{1}) should be rigorously justified. This work, which is technically the most difficult part of the paper,  is done in Section \ref{Sec4}, where the following result is established.   
\begin{thm}\label{mumiWWW} For every $r \in (0,1)$ and $b>0$ there exists a positive real number $c_*(r,b,K) \in [2\sqrt{1-r},2]$ such that system  (\ref{1}) has at least one positive monotone wavefront propagating with the speed $c$ if and only if $c \geq c_*(r,b,K)$.  
\end{thm}
Theorems \ref{mumiV} and \ref{mumiWWW} imply the following 
\begin{cor}  Assume that  $r, b>0$  satisfy $r+br\kappa(r,2\sqrt{1-r}) \leq 1$. 
Then $c_*(r,b,K)= 2\sqrt{1-r}$.  
\end{cor}
\begin{example}\label{ws}
The special kernels 
$$
K_w(s,y)=\frac{1}{\sqrt{4\pi s}}e^{-\frac{y^2}{4s}}\frac{1}{\tau}e^{-\frac{s}{\tau}}, \quad 
K_f(s,y)=\frac{1}{\sqrt{4\pi s}}e^{-\frac{y^2}{4s}}\frac{s}{\tau^2}e^{-\frac{s}{\tau}}.
$$
are called the weak delay kernel and the strong delay kernel, respectively. They are used to model  delayed systems with  nonlocal spatial effects, e.g. see \cite{KKNT-20} and references therein. In particular, the BZ system with the weak kernel was discussed in \cite{DQ}.   For the kernel $K_w$,  Theorem 1 in \cite{DQ}  establishes the existence of heteroclinic connections between the equilibria $(0,0)$ and $(1,1)$  when $r+b\leq 1$ and  $\tau>0$ is sufficiently small. Note that the monotonicity and positivity properties of these heteroclinics were not discussed in \cite{DQ}. 
Theorem \ref{mumiV} provides a significant improvement of this result. Indeed,  
a straightforward computation 
for the weak delay kernel yields 
$$
\kappa(r,c)=\frac{1}{1+\tau\lambda(c-\lambda)}=\frac{1}{1+\tau(1-r)}
$$
so that  conditions $r+br\kappa(r,c) \leq 1$, $c \geq 2\sqrt{1-r}$, take the  form
\begin{equation}\label{weakkernel}
r+\frac{br}{1+\tau(1-r)} \leq 1, \quad c \geq 2\sqrt{1-r}.
\end{equation}
For $\tau=0$ this result coincides with the ones from \cite[Theorem 3]{Ka2} and from \cite[Theorem 4.2]{yw}. If  $\tau > 0$ is sufficiently large, condition (\ref{weakkernel}) is  satisfied and the minimal wavefront is linearly determined. This goes in hand with the general observation that the time delay decreases the minimal speed of propagation in the monostable models with delays. 
Now, in Example \ref{wsc} of the last section we analyse numerically the minimal speed $c_*(b,3/4,\tau):= c_*(b,3/4,K_w)$ also for $b > 1/3+\tau/12$, i.e. for the case  when 
the first condition in (\ref{weakkernel}) is not necessarily met. By our computations, even for relatively small values of $b$ the minimal wavefronts seem to be non-linearly determined. Theorem \ref{dac3} below explains this phenomenon by establishing it rigorously  for sufficiently large $b$. 

Furthermore, we  get a similar result when we deal with the strong delay kernel. Calculations now provide 
$$
\kappa(r,c)=\frac{1}{(1+\tau\lambda(c-\lambda))^2}=\frac{1}{(1+\tau(1-r))^2}, 
$$
so that at least one wavefront for system (\ref{1}) considered with the kernel $K_f$ exists  by Theorem 4 if
\begin{equation*}\label{strongkernel}
r+\frac{br}{(1+\tau(1-r))^2} \leq 1, \quad c \geq 2\sqrt{1-r}. 
\end{equation*}
\end{example}

Finally, we consider the situations when a)   $b\to +\infty$ and b) $\kappa(r,c)\to +\infty$. In each  of them, $r+br \kappa(r,c)\gg1$ so that it is reasonable to expect that the respective minimal wavefronts are  not linearly determined. 

\begin{thm}  \label{dac3} Set $K_a(s,y)=K(s,y+a)$ and fix $b>0, r \in (0,1)$.  Then $c_*(r,b,K_a) \to 2$ as $a \to +\infty$. Next, fix some $r\in (0,1)$ and suppose that the kernel $K(s,y)$ is a continuous function of $s$ and $ y$. Then $c_*(r,b,K) \to 2$ as $b \to +\infty$. Thus the critical 
wavefront  is necessarily  pushed when either $b$ or $a$ is sufficiently large positive number. 
\end{thm} 
Theorem \ref{dac3} follows from Lemmas \ref{dac} and \ref{dac1}  which are proved   in Section \ref{Sec5} (together with other results of independent interest). 

%%%%%%%%%%%%%%%%%%%%%%%%%%%%%%%%%%%%%%%%%%%%%%%%%%%%%%%%%

\section{Monotonicity: Proof of Theorem \ref{mude}}\label{Sec2}
For the reader's convenience, the proof is divided into several simple steps. 

\vspace{1mm}

\noindent (a) {\sf Proof of the positivity of $\phi(t).$}

\noindent  Note that if $\phi(s)=0$ at some point $s$ then necessarily $\phi'(s)=0$ as a consequence of the non-negativity 
of $\phi(t)$.  Considering the first equation in  (\ref{3ade})  as a linear homogeneous non-autonomous equation for $\phi$, we obtain immediately from the uniqueness property that $\phi \equiv 0$, a contradiction. Hence, $\phi(t) >0$ for all $t \in \R$. 

\vspace{1mm}
 
\noindent(b) {\sf  Establishing upper and lower bounds for $\psi:$ $ 0 \leq \psi(t) \leq 1.$ }

%Next, we  find the bounds for  $\psi$. 
\noindent Since $\psi(-\infty) =0$, $\liminf_{t \to +\infty}\psi(t) >0$, the inequality $\psi(t') >1$  for
some $t'\in \R$ implies that either 

\noindent (A) $\psi'(t)>0, \ t \geq t' $ or 

\noindent (B) there exists point $s$ where $\psi'(s) =0, \ \psi''(s)
\leq 0,$ $\psi(s) >1$. 

{\noindent Due to the  positivity of $\phi$ (part (a)), the alternative (B) 
contradicts to the second equation of (\ref{3ade}). If the option (A) holds then 
$$
\psi''(t) = c\psi'(t) +b \phi(t)(\psi(t)-1) \geq b(\psi(t')-1)\left(\inf_{s\in [t',+ \infty)}\phi(s)\right)
>0, \quad t \geq t', 
$$
and therefore $\psi(t)$ should be unbounded, a contradiction. It follows that $\psi(t) \leq1$ for all $t.$
Using a similar argument one can show that $\psi(t) \geq 0$ for all $t.$ (The case (A) is 
totally analogous, case (B) reads as $\psi'(s) =0, \ \psi''(s) \geq 0,$ $\psi(s)
<0$, and the contradiction is achieved in the same way as above).

\vspace{1mm}

\noindent (c) {\sf  Improving  an upper  bound for $\psi:$ $ \psi(t) <1.$ }

\noindent Since $\theta(t)= 1-\psi(t) \in [0,1]$
satisfies
$
\theta''(t) - c\theta'(t) -b\theta(t)\phi(t) =0
$
and the equality $\theta(s)=0$ at some point $s$ necessarily would imply $\theta'(s)=0$, and, consequently, $\theta \equiv 0$ by the uniqueness theorem,  
we conclude that this equality  cannot happen so that $\theta(t) >0$ for all $t \in \R$. In other words, $\psi(t)\in [0,1)$  for all $t$. 

\vspace{1mm}

\noindent (d) {\sf  Proof of the monotonicity and positivity of $\psi$: $ \psi(t) \in (0,1), \  \psi'(t) >0, \ t \in \R$. }

\noindent Suppose now that $\psi'(s) =0$ at some point $s$ (note here that  $\psi(s)=0$ implies  $\psi'(s) =0$). 
Using (a) and (c), from the second equation in (\ref{3ade}) we get $\psi''(s)<0,$ so that $s$ is a strict local maximum
point and thus  $\psi(s)>0$. The positivity of $\psi$ follows. 
Next, since $\liminf_{t \to +\infty}\psi(t) >0$, there exists some $s' >s$ where $\psi'(s') < 0, \
\psi''(s')= 0,$ $\psi(s') \in (0,1)$. This again contradicts to
the second equation of (\ref{3ade}).

\vspace{1mm}

\noindent (e) {\sf  Proof of the convergence of $\psi$: $ \psi(+\infty) =1$. }

\noindent Let $s_k\to -\infty, t_k \to +\infty$ be such that $ \psi'(s_k), \psi'(t_k) \to 0$ as $k \to +\infty$. 
Integrating  the second equation in (\ref{3ade}) on $[s_k, t_k]$ and then taking limit as $k \to +\infty$, 
we find that  
$$ c = b \int_\R\phi(t)(1-\psi(t))dt. $$
This shows that $ \psi(+\infty) =1$.  
\vspace{1mm}

\noindent (f) {\sf  Establishing  an upper  bound for $\phi:$ $ \phi(t) < 1.$ }

\noindent Suppose by contradiction that $\phi(t') \geq 1$ for some $t'\in {\R}$. Then either 

\noindent (C) $\phi'(t) >0$ for all $t\geq t'$ or 

\noindent (D) there exists a local maximum point $s$ where $\phi'(s) =0, \
\phi(s)
\geq 1,\ \phi''(s) \leq 0$. 

\noindent Since, by (d),  $(K\star\psi)(s)\in (0,1)$, the alternative (D) yields 
a contradiction with the first equation of (\ref{3ade}).  If the option (C) holds, then $(K\star\psi)(+\infty) =1$ (by (e)) and thus
$$
\liminf_{t \to +\infty}\phi''(t) = \liminf_{t \to +\infty}\left(c\phi'(t) + \phi(t)(\phi(t)-1 + r(1 - (K\star\psi)(t))\right) \geq  $$
$$\liminf_{t \to +\infty}\left(\phi(t)(\phi(t)-1 + r(1 - (K\star\psi)(t))\right)= \phi(+\infty)(\phi(+\infty)-1) >0, \quad t \to +\infty, 
$$
and therefore $\phi(t)$ should be unbounded. 

\vspace{1mm}
}

\noindent (g) {\sf  Proof of the monotonicity of $ \phi(t).$ }

\noindent Suppose now that $\phi'(s) =0, \phi(s) \in (0,1)$ at some point
$s$. First we consider the case when additionally $\phi''(s) =0,$
so that $1 - r- \phi(s)+ r(K\star\psi)(s)=0$. Differentiating the first
equation in  (\ref{3ade}), we  find that $\phi'''(s) =
- r(K\star\psi')(s)\phi(s) <0$.

This implies
that  $\phi(t) > \phi(s)$ for all $t < s$ close
to $s$. As a consequence, there exists $s' < s$ such that
$\phi'(s') =0, \phi''(s') \leq 0, $ $  1 > \phi(s') > \phi (s), \ (K\star\psi)(s')\leq 
(K\star\psi)(s).$ But then
$$
0=1 - r- \phi(s)+ r(K\star\psi)(s)>  1 - r- \phi(s')+ r(K\star\psi)(s'),
$$
which yields a contradiction: 
$0=\phi''(s') - c\phi'(s') + \phi(s') (1 - r- \phi(s')+
r(K\star\psi)(s'))<0.
$

Similarly, if $\phi''(s)<0,$ then  $s$ is a local maximum point and
$1 - r- \phi(s)+ r(K\star \psi)(s)>0$. Since $\liminf_{t \to +\infty}\phi(t) >0$,
there is  some $s'
>s$ where $\phi'(s') <0, \ \phi''(s') = 0,$ $0< \phi(s')< \phi(s),
\ (K\star\psi)(s)\leq (K\star\psi)(s')$ and therefore
\begin{eqnarray*}
0&<&1 - r- \phi(s)+ r(K\star\psi)(s)<  1 - r- \phi(s')+ r(K\star\psi)(s'),\\
0&=&\phi''(s') - c\phi'(s') + \phi(s') (1 - r- \phi(s')+
r(K\star\psi)(s'))>0,
\end{eqnarray*}
a contradiction.

Finally, let  $\phi'(s_1) =0$  and
$\phi''(s_1)
>0$ at some point $s_1$. Then 
there exists a local maximum point $s < s_1$ where $\phi'(s) =0$
and $\phi''(s) \leq 0$. However, this possibility was already
rejected. \hfill $\square$

\section{Proof of Theorem \ref{mumiV}}\label{Exi}
Given  $c \geq 2\sqrt{1-r}$, let $0<\lambda=\lambda_r(c) \leq \mu=\mu_r(c)$ denote the zeros of the characteristic polynomial  $\chi_r(z,c)=z^2-cz +1-r$. 

\vspace{2mm}

\noindent\underline{Step 1: $c > 2\sqrt{1-r}$}.   
For
$B = \min\{-(1+r+b), -4\lambda^2\}$, 
consider the  operators
$$
      {\mathcal F}_1(\phi, \psi)(t)=   \phi(t) (1 - r- B-\phi(t)+ r(K\star\psi)(t)), \quad
     {\mathcal F}_2(\phi, \psi)(t) = b \phi(t)(1-\psi(t)) - B\psi(t), 
$$
Note that  ${\mathcal F}_1, {\mathcal F}_2$  are
monotone in the sense that ${\mathcal F}_j(\phi_1, \psi_1)(t) \leq
{\mathcal F}_j(\phi_2, \psi_2)(t),$ $j=1,2, $ if $0\leq
\phi_1(t) \leq \phi_2(t)\leq 1, \  0\leq \psi_1(t) \leq
\psi_2(t)\leq 1, \ t \in {\R} $.
 Let $z_1 <0< z_2$ be the real roots
of the auxiliary equation $z^2-cz+B=0$.
Then  each bounded solution $(\phi, \psi)$ of the differential equations in  (\ref{3ade})  satisfies the system 
\begin{equation}\label{pfoin}
\phi(t)={\mathcal N}_1(\phi, \psi)(t), \  \psi(t)= {\mathcal
N}_2(\phi, \psi)(t), \quad {\rm where}
\end{equation}
$$
{\mathcal N}_j(\phi, \psi)(t):=\frac{1}{z_2 - z_1}
\left(\int^t_{-\infty} e^{z_1 (t-s)}{\mathcal F}_j(\phi, \psi)(s)ds + \int_t^{+\infty}e^{z_2
(t-s)}{\mathcal F}_j(\phi, \psi)(s)ds\right). 
$$
Conversely, each positive strictly monotone bounded solution $(\phi,
\psi)$ of  (\ref{pfoin}) yields a wavefront for 
(\ref{3ade}). It is clear that the operators ${\mathcal N}_j, j=1,2$  are also monotone.  As a consequence, each ${\mathcal N}_j(\phi, \psi)(t)$  increases in $t$ if both $\phi, \psi : {\bf R} \to [0,1]$ are increasing functions: 
$$
{\mathcal N}_j(\phi(\cdot), \psi(\cdot))(t-h) = {\mathcal N}_j(\phi(\cdot-h), \psi(\cdot-h))(t) \leq {\mathcal N}_j(\phi(\cdot), \psi(\cdot))(t), \quad h \geq 0.  
$$

%%%%%%%%%%%%%%%%%%%%%%%%%%%%%%%%%%%%%%%%%%%%%%%%%

\begin{lem} \label{Le2} Suppose that $r \in (0,1), \ c > 2\sqrt{1-r}$. Set $b_1:= b/(1-r)$. Then \mbox{$rb_1\kappa(r,c) \leq 1$} implies that 
$\Phi_+(t)= \min\{1, e^{\lambda t}\}$, $\Psi_+(t)= \min\{1, b_1e^{\lambda t}\}$, %$\lambda \in (D_1,D_2)$,  
 satisfy
the inequalities 
\begin{equation}\label{PiP}
\Phi_+^{(1)}(t) := {\mathcal N}_1(\Phi_+,
\Psi_+)(t) \leq  \Phi_+(t), \quad \Psi_+^{(1)}(t):= {\mathcal N}_2(\Phi_+,
\Psi_+)(t)\leq \Psi_+(t).
\end{equation}
\end{lem}
\begin{proof} Set $(\phi_+(t), \psi_+(t))=(1,b_1)e^{\lambda t}$. Then the assumptions of the lemma imply that 
\begin{equation}
\label{3fs} \begin{array}{ll}
{\cal D}_1(\phi_+,\psi_+)(t):= \phi_+''(t) - c\phi_+'(t) + \phi_+(t) (1 - r- \phi_+(t)+ r(K\star\psi_+)(t)) \leq 0,
      \\
{\cal D}_2(\phi_+,\psi_+)(t):=     \psi_+''(t) - c\psi_+'(t) +b \phi_+(t)(1-\psi_+(t)) < 0.
\end{array}%
\end{equation}
Indeed, ${\cal D}_1(\phi_+,\psi_+)(t) = -e^{2\lambda t}(1-rb_1\kappa(r,c))\leq 0$, \ ${\cal D}_2(\phi_+,\psi_+)(t) = -bb_1e^{2\lambda t}<0$. That is, $(\phi_+(t), \psi_+(t))$ is a regular upper solution to the system (\ref{3ade}) and therefore 
$$  {\mathcal N}_1(\phi_+,
\psi_+)(t)  \leq  \phi_+(t), \quad {\mathcal N}_2(\phi_+,
\psi_+)(t) \leq \psi_+(t).$$
Note that the improper integrals  ${\mathcal N}_j(\phi_+,
\psi_+)(t)$ are convergent for each $t\in \R$, $j=1,2$, because $2\lambda < z_2$ due to our choice of $B$.  
On the other hand, % since $(1,1)$ is an equilibrium for  (\ref{3ade}), 
we have that 
$$ {\mathcal N}_1(1,1)(t)  \leq  1, \quad {\mathcal N}_2(1,1)(t) \leq 1.$$
Thus, using the monotonicity properties of ${\mathcal N}_j$, we find that, for every $t \in \R$, 
$$  {\mathcal N}_1(\Phi_+,
\Psi_+)(t) \leq {\mathcal N}_1(\phi_+,
\psi_+)(t) \leq \phi_+(t), \quad {\mathcal N}_1(\Phi_+,
\Psi_+)(t) \leq {\mathcal N}_1(1,
1)(t) \leq 1,$$
which implies the first inequality in (\ref{PiP}). The proof of the second inequality in (\ref{PiP}) is similar. 
\qed
\end{proof}
Set now
$ \Psi_-(t) \equiv 0$ and let  $\Phi_-(t), \ t \in {\R},$ be a unique positive (up to a shift) wavefront solution to the usual KPP-Fisher equation
$$
\phi''(t) - c\phi'(t) + \phi(t) (1 - r- \phi(t)) = 0,\
\phi(-\infty)=0, \ \phi(+\infty) = 1-r >0.
$$
It is well known  that $\Phi_-(t)$ is strictly
increasing and that, without loss of generality, we can assume that 
 $\Phi_-(t) \leq e^{\lambda t},$ $ t \in
\R.$  Consequently, 
\begin{equation}\label{pYp}
\Phi_-(t)\leq 
\Phi_+(t), \quad \Psi_-(t)\leq 
\Psi_+(t), \quad t \in \R,
\end{equation}
and 
\begin{equation}
\label{3fs} \begin{array}{ll}
\Phi_-''(t) - c\Phi_-'(t) + \Phi_-(t) (1 - r- \Phi_-(t)+ r(K\star\Psi_-)(t)) =0,
      \\
 \Psi_-''(t) - c\Psi_-'(t) +b \Phi_-(t)(1-\Psi_-(t))> 0.
\end{array}%
\end{equation}
That is, $(\Phi_-(t), \Psi_-(t))$ is a regular lower solution to the system (\ref{3ade}) and therefore 
$$
\Phi_-(t) \leq  \Phi_-^{(1)}(t):= {\mathcal N}_1(\Phi_-,
\Psi_-)(t); \quad  
\Psi_-(t)  < \Psi_-^{(1)}(t):= {\mathcal N}_2(\Phi_-,
\Psi_-)(t). $$
Hence, applying the monotone integral operators ${\mathcal N}_j$ to (\ref{pYp}), we get
\begin{eqnarray*}
\Phi_-(t) &\leq & \Phi_-^{(1)}(t):= {\mathcal N}_1(\Phi_-,
\Psi_-)(t) \leq  {\mathcal N}_1(\Phi_+,
\Psi_+)(t)=: \Phi_+^{(1)}(t)  \leq  \Phi_+(t), 
 \\
\Psi_-(t)  &< &\Psi_-^{(1)}(t):= {\mathcal N}_2(\Phi_-,
\Psi_-)(t)  \leq {\mathcal N}_2(\Phi_+,
\Psi_+)(t)=: \Psi_+^{(1)}(t) \leq \Psi_+(t).
\end{eqnarray*}
Iterating this procedure, we obtain four sequences of positive  monotone continuous functions 
\begin{equation} \label{apli} \hspace{7mm}
 \Psi_-^{(n+1)}(t) =  {\mathcal N}_2(\Phi_-^{(n)}, \Psi_-^{(n)})(t), \ \Phi_-^{(n+1)}(t) =  {\mathcal N}_1(\Phi_-^{(n)}, \Psi_-^{(n)})(t), \  n =1,2, \dots,
\end{equation}

\vspace{-7mm} 

$$
\hspace{-7mm} {\rm and} \ \   \Psi_+^{(n+1)}(t) =  {\mathcal N}_2(\Phi_+^{(n)}, \Psi_+^{(n)})(t), \ \Phi_+^{(n+1)}(t) =  {\mathcal N}_1(\Phi_+^{(n)}, \Psi_+^{(n)})(t), \  n =1,2, \dots,
$$
The sequences  $\Psi_-^{(n)}$, $\Phi_-^{(n)}$  are strictly increasing and   $\Psi_+^{(n)}$, $\Phi_+^{(n)}$   are strictly decreasing.  Set $\Phi = \lim \Phi_-^{(n)}, \
\Psi = \lim \Psi_-^{(n)}$, then 
\begin{equation} \label{zer}
\Phi_- \leq  \Phi \leq \Phi_+, \ \Psi_-\leq \Psi \leq  \Psi_+.
\end{equation}
Finally, a straightforward application of the Lebesgue's dominated convergence theorem %to (\ref{apli})
shows that the pair $(\Phi, \Psi)$  satisfies system (\ref{pfoin}). Clearly, $\Phi(t), \Psi(t), \ t \in \R,$ are positive bounded monotone functions meeting the boundary conditions $\Phi(-\infty) = \Psi(-\infty)=0$ because of (\ref{zer}).  Thus  the values of $\Phi(+\infty), \Psi(+\infty)$ are finite and positive.  A standard argument based on the Barbalat lemma (cf. \cite{wz}) shows that $(\Phi(+\infty),\Psi(+\infty))$ is a positive equilibrium to  the system of  differential equations (\ref{3ade}).  This completes the proof of Theorem \ref{mumiV}, when $c > 2\sqrt{1-r}$. 

\vspace{2mm}

\noindent\underline{Step 2: $c =c_*:=2\sqrt{1-r}$}.  Consider a strictly decreasing sequence $c_j \to c_*$ and strictly increasing sequence $b_j \to b$ such that 
$r+b_jr\kappa(r,c_j) < 1$ for all large $j$. The existence of such sequences follows easily from the continuous dependence of  $\kappa(r,c)$ on the variable $c$. Therefore Step 1 guarantees that the system (\ref{1}) considered with $c=c_j, b = b_j$ has at least one positive monotone wavefront
$(\phi_j(x+c_jt), \psi_j(x+c_jt))$ normalized by the condition $\phi_j(0)=0.5$. Furthermore, the functional sequences $(\phi'_j(t), \psi'_j(t))$ are uniformly bounded on $\R$. Indeed, we have that $\liminf_{t \to \pm \infty}\phi'_j(t)= \liminf_{t \to \pm \infty}\psi'_j(t)=0$.  Therefore, if  $\alpha=\sup_{t \in \R}\phi'_j(t)$,  then there exists a sequence (possibly finite) $s_k$ such that 
$
\phi'_j(s_k)
$ converges (or is equal) to $\alpha$ and 
$
\phi''_j(s_k)=0. 
$
Since, in virtue of (\ref{3ade}), 
\begin{equation}\label{psk}
\phi'_j(s_k) =c_j^{-1}\phi_j(s_k) (1 - r- \phi_j(s_k)+ r(K\star\psi_j)(s_k))  \leq (4c_j)^{-1}\leq (4c_*)^{-1},
\end{equation}
we obtain that $\sup_{t \in \R}\phi'_j(t)\leq 1/(4c_*)$ for every $j$. Similarly, $\sup_{t \in \R}\psi'_j(t)\leq b/c_*$,  $j\in \N$. 
Using these estimates of derivatives, from the system  (\ref{3ade}) we obtain the uniform boundedness of $(\phi''_j(t), \psi''_j(t))$. But then, after   differentiating (\ref{3ade}) with respect to $t$, we get the uniform boundedness of $(\phi'''_j(t), \psi'''_j(t))$.  All this allows us  to establish, with the help of the Ascoli-Arzel\`a theorem, the existence of a subsequence  of functions $(\phi_{j_k}(t), \psi_{j_k}(t))$, uniformly converging in $C^2$-norm on compact subsets of $\R$ to a pair of $C^2$-smooth non-decreasing functions $(\phi_{*}(t), \psi_{*}(t)),$ 
such that $\phi_{*}(0)=0.5$. Such a convergence assures that the limit functions satisfy  differential equations in  (\ref{3ade}), while their limit values at $\pm\infty$ belong to the set of equilibria for (\ref{3ade}).  The latter fact and the relation  $\phi_{*}(0)=0.5$ imply that $\phi_{*}(-\infty)=0, \ \psi_{*}(+\infty) = \phi_{*}(+\infty)=1$. To prove that $\psi_{*}(-\infty)=0$, it suffices to apply Lemma \ref{mda}. By this lemma, 
 $\psi_j(t) < \sqrt{M \phi_j(t)}, \ t \in \mathbb{R}$ so that  $\psi_*(t) \leq \sqrt{M \phi_*(t)}, \ t \in \mathbb{R},$
and therefore $\psi_*(-\infty)=0$.  \qed

%%%%%%%%%%%%%%%%%%%%%%%%%%%%%%%%%%%%%%%%%%%%%%%%%%%%%%%%%%%%%%%%%%%%%

\section{Proof of Theorem \ref{mumiWWW}}\label{Sec4}
Model (\ref{1}) in the limiting case $r=0$ semi-splits, i.e. the first equation is independent on $v.$  It actually coincides with the 
KPP-Fisher equation,  which has a unique monotone wavefront $u=\phi_c(x+ct)$ for each velocity $c \geq 2$.  This suffices to deduce the existence of the accompanying monotone front $v=\psi_c(x+ct)$ from the second equation of (\ref{1}). 

In this section, we show that   system (\ref{1}) with  $r \in (0,1)$ has the same properties. Actually,  equations (\ref{1}) with $r=0$ can be regarded as a starting comparison system for the case when 
$r >0$. Particularly, we  fix an arbitrary  $c >2$ and use the first component of the basic upper solution 
$$
\phi_+(t)= \left\{
\begin{array}{ll}
    e^{\lambda_0t},
    &    t \leq 0,\\
     1, & t \geq 0, \\   \end{array}%
\right.
$$

\noindent where    $\lambda_0:=\lambda_0(c), \ c > 2,$  while the second component $\psi_+(t)$ is defined as the unique strictly monotone  solution of the boundary value problem 
\begin{equation}\label{3adek}
     \psi''(t) - c\psi'(t) +b_1 \phi_+(t)(1-\psi(t)) =0,  \quad    \psi(-\infty)=0, \  \psi(+\infty)=1, 
\end{equation}
with  $b_1>b$.  The existence of such a solution is established in Lemma \ref{OWL} of Appendix. {We have that  $
\psi_+(t)= b_1 e^{\lambda_0t}(1+  o(e^{\gamma t})), \ t\to-\infty$, with an appropriate $\gamma >0$,  see e.g. \cite[Proposition 6.1]{MP}.
As a consequence, for some $B_1 \geq b_1$, 
\begin{equation}\label{B}
0 < \psi_+(t) \leq   B_1 e^{\lambda_0t}, \quad t \in \R. 
\end{equation}}
Since $b_1>b$, we also find that  
\begin{equation}\label{3adekv}
    \psi_+''(t) - c\psi_+'(t) +b \phi_+(t)(1-\psi_+(t)) <0,  \quad t \in  \R.  
\end{equation}
In addition, we also have that 
\begin{equation}\label{19}
 \phi_+''(t) - c\phi_+'(t) + \phi_+(t) (1 - r- \phi_+(t)+ r(K\star \psi_+)(t)) =
\end{equation}
$$
= \left\{
\begin{array}{ll}
    -e^{2\lambda_0t} -re^{-\lambda_0t}(1-(K\star \psi_+)(t)) <0,
    &    t \leq 0,\\
     -r(1-(K\star \psi_+)(t)) <0, & t \geq 0. \\   \end{array}%
\right.
$$
Hence, the functions $\phi_+, \psi_+$ satisfy the basic differential inequalities  for super-solutions and 
they also have the desired behaviour on $\R_+$. However, since $\lambda_r=\lambda_r(c) < \lambda_0(c)=\lambda_0$ for $r>0$, they decay too rapidly at $-\infty$. We will overcome
this  drawback by adding small correction terms to $\phi_+(t), \psi_+(t)$ for $t \leq -1$.   

\vspace{2mm} 

At the first stage, let us consider the situation when $K(s,y)$ has compact support, say supp\,$K \subset [0,m]\times [-m,m]$. Then 
$$
(K\star e^{j\lambda_r \cdot})(t)= e^{j\lambda_rt}\kappa_j(r,c), \quad \kappa_j(r,c):= \int_0^{m}\int_{-m}^m K(s,y)e^{-j\lambda_r(cs+y)}dy\;ds. 
$$

Let $\epsilon >0$ be a small number and  let a non-increasing function $\eta \in C^\infty(\R)$ be such that $\eta(t)=1,$ $ t \leq -2$, $\eta(t)=0, t \geq -1$, $\eta'(t) < 0,$ $t \in (-2,-1)$.  With $k \in \N$ being the maximal integer such that $k\lambda_r \leq  \lambda_0$ (so that, assuming that $k>1$, we have $\chi_r(j\lambda_r,c)<0$ for $j=2,\dots,k$), consider 
the 
functions 
$$
\Phi_+(t)= \phi_+(t) + \phi_\epsilon(t),\quad  \phi_\epsilon(t):= \eta(t)\sum_{j=1}^ka_j\epsilon^je^{j\lambda_rt},$$
$$  \Psi_+(t)= \psi_+(t) +\psi_\epsilon(t), \quad \psi_\epsilon(t):= \eta(t)\sum_{j=1}^kb_j\epsilon^je^{j\lambda_rt}, 
$$
where $a_1=1$, $b_1=b/(1-r)$,  and for  $j >1$ %then  other coefficients $a_j, b_j$ are defined as
\begin{equation}\label{aj}
a_j =a_j(r,c):=\frac{\sum_{p+q=j}a_p(a_q-r\kappa_q(r,c)b_q)}{\chi_r(j\lambda_r,c)},
\end{equation}
\begin{equation}\label{bj}
b_j =b_j(r,c):=\frac{-ba_j}{(j\lambda_r)^2-c(j\lambda_r)}= \frac{ba_j}{ 1-r-\chi_r(j\lambda_r,c)}. 
\end{equation}
Clearly, $\Phi_+(t) =\phi_+(t),\ \Psi_+(t)= \psi_+(t)$ for $t \geq -1$, so that 
$${\cal D}_1(\Phi_+,\Psi_+)(t)= {\cal D}_1(\phi_+,\psi_+)(t) <0$$
 for all $t \geq m(1+c)$ (and ${\cal D}_1$ defined by \eqref{3fs}).

 Furthermore, since $a_1, b_1$ are positive, we have that $\phi_\epsilon(t), \psi_\epsilon(t) \geq 0, \ t \in \R$, and 
$$
\phi_\epsilon(t) >0,\ \phi_\epsilon'(t)>0,\ \psi_\epsilon(t) >0,\ \psi_\epsilon'(t) >0, \quad t \leq -2; \qquad \Psi_+(t) <1, \ t \in \R, 
$$
for all  sufficiently small $\epsilon >0$. 
Hence, for all $t \leq -3-m(1+c)$, 
$$
{\cal D}_1(\Phi_+,\Psi_+)(t) =  \Phi_+''(t) - c\Phi_+'(t) + \Phi_+(t) (1 - r- \Phi_+(t)+ r(K\star \Psi_+)(t)) =
$$
$$
e^{\lambda_0t} (- e^{\lambda_0t}-r(1-(K\star \psi_+)(t))) + \sum_{j=1}^ka_j\chi_r(j\lambda_r,c)\epsilon^je^{j\lambda_rt}
-{\left(\sum_{j=1}^ka_j\epsilon^je^{j\lambda_rt}\right)^{2}}+
$$
{
$$
e^{\lambda_0t}\sum_{j=1}^k(rb_j\kappa_j(r,c) -2a_j) \epsilon^je^{j\lambda_rt} +r(K\star \psi_+)(t)\sum_{j=1}^ka_j\epsilon^je^{j\lambda_rt}+ $$
$$+r\sum_{j=1}^ka_j\epsilon^je^{j\lambda_rt}\sum_{j=1}^kb_j\kappa_j(r,c)\epsilon^je^{j\lambda_rt}=:e^{\lambda_0t}(\Delta_1(t,\epsilon)+\Delta_2(t,\epsilon)) =:e^{\lambda_0t}\Delta(t,\epsilon),  
$$
where, in view of (\ref{aj}), (\ref{B}),  respectively, 
$$
\Delta_1(t,\epsilon):= - e^{\lambda_0t}-r(1-(K\star \psi_+)(t))+ \epsilon e^{\lambda_rt}S_{k-1}(\epsilon e^{\lambda_rt}) +\epsilon^{k+1}e^{((k+1)\lambda_r-\lambda_0)t}T_{k-1}(\epsilon e^{\lambda_rt}),
$$
$$
|\Delta_2(t,\epsilon)|:=\left|e^{-\lambda_0t}r(K\star \psi_+)(t)\sum_{j=1}^ka_j\epsilon^je^{j\lambda_rt}\right|\leq B_1r\kappa_1(0,c)\left|\sum_{j=1}^ka_j\epsilon^je^{j\lambda_rt}\right|, \ t \in \R, 
$$
and  $S_{k-1}, T_{k-1}$ are real polynomials of degree $k-1$.  

Note that $\Delta(t,0)<0$, $\Delta(-\infty,\epsilon)=-r$, 
and $\lim_{\epsilon \to 0^+}\Delta(t,\epsilon) = \Delta(t,0)$ uniformly on $(-\infty, -3-m(1+c)]$. }

Therefore ${\cal D}_1(\Phi_+,\Psi_+)(t) <0,$ $t \in (-\infty, -3-m(1+c)]$,  for all small $\epsilon >0$.  In addition, since $\phi_\epsilon(t),  \psi_\epsilon(t)$ are $C^\infty$-smooth in $\epsilon, t$, we have that 
$$\lim_{\epsilon\to 0^+}{\cal D}_1(\Phi_+,\Psi_+)(t) = {\cal D}_1(\phi_+,\psi_+)(t),$$
 uniformly  on $t \in [-3-m(c+1),m(1+c)]\setminus\{0\}$. In this way,  there exists $\epsilon_0>0$ such that 
$$\Phi_+'(0^+)-  \Phi_+'(0^-)=\phi_+'(0^+)-  \phi_+'(0^-)>0, \quad {\cal D}_1(\Phi_+,\Psi_+)(t)  <0, \quad t \in \R\setminus\{0\}, \quad \epsilon \in (0,\epsilon_0].$$

Similarly, using (\ref{bj}), we find that 
\begin{equation}\label{22}
{\cal D}_2(\Phi_+,\Psi_+)(t) = \Psi_+''(t) - c\Psi_+'(t) + b\Phi_+(t) (1 -  \Psi_+(t)) =
\end{equation}
$$
    \psi_+''(t) - c\psi_+'(t) +b \phi_+(t)(1-\psi_+(t)) + \sum_{j=1}^k(b_j((j\lambda_r)^2-c(j\lambda_r))+ba_j) \epsilon^je^{j\lambda_rt} $$
    $$-b\Phi_+(t)\psi_\epsilon(t)- b\psi_+(t)\phi_\epsilon(t)\leq     \psi_+''(t) - c\psi_+'(t) +b \phi_+(t)(1-\psi_+(t)) <0, \quad t \in \R. 
$$
In this way,  $(\Phi_+(t), \Psi_+(t))$ is a super-solution for system (\ref{3ade}) and therefore 
$$  {\mathcal N}_1(\Phi_+,
\Psi_+)(t)  \leq  \Phi_+(t), \quad {\mathcal N}_2(\Phi_+,
\Psi_+)(t) \leq \Psi_+(t),$$
e.g. see \cite[Lemma 18]{TPTa} for details. 

 Next, we  take  $\Psi_-(t) \equiv 0$ and $\Phi_-(t)$ defined in Section \ref{Exi} as suitable lower solutions. Since $\Phi_-(t)$, $\Phi_+(t)$ 
are monotone and $\Phi_-(t)/\Phi_+(t)$ converges to a finite positive number at $-\infty$, without loss of generality, we can assume that relations (\ref{pYp}) are satisfied after an appropriate translation of $(\Phi_+(t), \Psi_+(t))$.  As we know from Section \ref{Exi}, this implies the existence of monotone wavefront $(\Phi(t), \Psi(t))$ propagating with the given speed  $c> 2$. 

\vspace{2mm} 

At the second stage, we consider $K(s,y)$  having non-compact support and $c\geq 2$. Let $K_m(s,y)$ be a sequence of 
normalised kernels such that supp\,$K_m \subset [0,m]\times [-m,m]$ and $K_m\to K$ uniformly on compact sets.  Since we include the limit case $c=2$, we also consider a strictly decreasing sequence $\{c_m\}$ converging to $c$. 
  Then the above argumentation guarantees that system (\ref{3ade}) considered with the speed $c_m>2$  and the interaction kernel $K_m$ has at least one positive monotone wavefront
$(\phi_m(x+c_mt), \psi_m(x+c_mt))$ normalised by the condition $\phi_m(0)=0.5$.

Furthermore, the functional sequences $\phi'_m(t), \psi'_m(t)$ are uniformly bounded on $\R$ by $(1+b)/2$, cf. (\ref{psk}). 
Using these estimates of derivatives, we obtain from  system  (\ref{3ade})  the uniform boundedness of $(\phi''_m(t), \psi''_m(t))$. But then, after   differentiating (\ref{3ade}) with respect to $t$, we get the uniform boundedness of $(\phi'''_m(t), \psi'''_m(t))$.  All this allows us  to establish, with the help of the Ascoli-Arzel\`a theorem, the existence of a subsequence  of functions $(\phi_{m_k}(t), \psi_{m_k}(t))$, uniformly converging in $C^2$-norm on compact subsets of $\R$ to some $C^2$-smooth non-decreasing functions $(\phi_{*}(t), \psi_{*}(t)),$ 
such that $\phi_{*}(0)=0.5$.  Arguing now as in Step 2 of the proof of Lemma \ref{Le2}, we find that $(\phi_{*}(t), \psi_{*}(t))$ is actually a  positive monotone wavefront for (\ref{1}) propagating at the speed $c\geq 2$. 

\vspace{2mm} 
{
Hence, the above two stages of our analysis have led  to the following partial conclusion:  

{ \it For each triple of parameters $r \in (0,1), b>0$, $c \geq 2$  there exists 
at least one positive monotone wavefront for (\ref{1}) propagating with the speed $c$}. 

This assertion says that 
$c_*(r,b,K) \leq 2$. To complete the proof of Theorem \ref{mumiWWW}, we still should establish that the set of all admissible speeds  for (\ref{3ade}) is a connected closed interval.   This work is done in the remainder of this section. }

%%%%%%%%%%%%%%%%%%%%%%%%%%%%%%%%%%%%%%%%%%%%%%%%%%%%%%%%%%%%%%%%%%%%%%
\vspace{2mm} 

So, let $(\phi(t),\psi(t))$ be a monotone wave 
to  system (\ref{3ade}) propagating with the speed $c$. 

\vspace{2mm}

\noindent (a) We claim that for every $\tilde c>c$ there exists $\sigma >1$ such that, for all $t \in {\R}$,
the pair $(\phi_\sigma,\psi_\sigma) = \sigma(\phi,\psi)$ satisfies the inequalities 
\begin{equation}\label{3s}\left\{
\begin{array}{ll}
{\cal D}_1(\phi_\sigma,\psi_\sigma)(t):=  \phi_\sigma''(t) - \tilde c\phi_\sigma'(t) + \phi_\sigma(t) (1 - r- \phi_\sigma(t)+ r(K\star_{\tilde c}\psi_\sigma)(t)) <0,
    &    \\
   {\cal D}_2(\phi_\sigma,\psi_\sigma)(t):=  \psi_\sigma''(t) - \tilde  c\psi_\sigma'(t) +b \phi_\sigma(t)(1-\psi_\sigma(t)) < 0, 
&
\end{array}%
\right.
\end{equation}
where we use the notation$$
(K\star_\alpha \psi)(t):= \int_0^{+\infty}\int_\R K(s,y)\psi(t-\alpha s-y)dy\; ds. 
$$
Alternatively, 
\begin{equation}\label{3w}\left\{
\begin{array}{ll}
(\phi'(t))^{-1}\phi(t) (- \phi(t) -\frac{\gamma(t)r}{\sigma-1}+ r(K\star_{\tilde  c} \psi)(t)) < \frac{\tilde c-c}{\sigma-1},
    &    \vspace{2mm} \\ 
    -b(\sigma-1)\phi(t)\psi(t) < (\tilde  c-c)\psi'(t), &
\end{array}%
\right.
\end{equation}
where, from the monotonicity  of $\psi$,  $\gamma(t)= (K\star_c\psi)(t)- (K\star_{\tilde c}\psi)(t)= $ $$=\int_0^{+\infty}\int_\R K(s,y)\left(\psi(t-cs-y)-\psi(t- \tilde cs-y)\right)dy\;ds \geq 0, \quad t \in \R, \quad \gamma(\pm\infty)=0.  $$
 Since $\phi(t), \psi(t), \phi'(t), \psi'(t) >0$ for all $t \in {\R}$, the second inequality in \eqref{3w} is always satisfied and 
 it suffices to prove  that,  for some appropriate $\sigma >1$, 
$$\Psi(t):=(\phi'(t))^{-1}\phi(t) (- \phi(t) + r(K\star_{\tilde  c} \psi)(t)) < \frac{\tilde  c-c}{\sigma-1}, \quad t \in \R.$$

From Proposition \ref{mudas}, we know that $(\phi'(t))^{-1}\phi(t)$ has a finite limit at $t = -\infty$ when $r \in (0,1)$.   Thus
$\Psi(-\infty)=0, \  \Psi(+\infty) = - \infty$  so that  if $\sigma
>1$ is close to $1$, then $\Psi(t) <
(\sigma-1)^{-1}(\tilde  c-c)$ for all $t \in {\R}$.

Therefore  we  found appropriate $\sigma >1,$ for which the inequalities \eqref{3s} are satisfied.

\vspace{2mm}
 
\noindent (b) We proceed by establishing the existence of wavefronts propagating with the speed $\tilde  c$. Here our proof follows closely the arguments developed in Section \ref{Sec4}, beginning from relation (\ref{19}). Again, we will first assume that $K$ has 
a compact support contained in the rectangle $[0,m]\times [-m,m]$. Since the asymptotic behaviour at $-\infty$ of both pairs $(\phi_\sigma,\psi_\sigma)$ and 
$(\phi,\psi)$ is determined by one of the eigenvalues $\mu_r(c) \geq \lambda_r(c)$ which are strictly bigger than $\lambda_r':= \lambda_r(\tilde  c)$, we again will use the correcting terms $\phi_\epsilon(t), \psi_\epsilon(t)$ from Section \ref{Sec4} 
and define the upper solutions by {
$$
\Phi_+(t)= \phi_\sigma(t) + \phi_\epsilon(t),\quad  \phi_\epsilon(t):= \eta(t)\sum_{j=1}^ka_j(r,\tilde c)\epsilon^je^{j\lambda_r't},$$
$$  \Psi_+(t)= \psi_\sigma(t) +\psi_\epsilon(t), \quad \psi_\epsilon(t):= \eta(t)\sum_{j=1}^kb_j(r,\tilde c)\epsilon^je^{j\lambda_r't}, 
$$
where $k$ is the maximal integer such that $k\lambda_r'=k\lambda_r(\tilde c) \leq  \mu_r(c)$. Note that, whenever $k>1$, we have $\chi_r(j\lambda_r',\tilde  c)<0$ for $j=2,\dots,k$ and therefore the numbers $a_j' =a_j(r,\tilde c),$ $b_j' =b_j(r,\tilde c)$ are also well defined for $j >1$.} Then,  for all $t \leq -3-m(1+ c)$ and sufficiently small $\epsilon$,  by using (\ref{3s}), we get 
$$
{\cal D}_1(\Phi_+,\Psi_+)(t) =  \Phi_+''(t) - \tilde  c\Phi_+'(t) + \Phi_+(t) (1 - r- \Phi_+(t)+ r(K\star_{\tilde  c} \Psi_+)(t)) =
$$
$$
{\cal D}_1(\phi_\sigma,\psi_\sigma)(t)+ \sum_{j=1}^ka_j'\chi_r(j\lambda_r',\tilde  c)\epsilon^je^{j\lambda_r't}
-{\left(\sum_{j=1}^ka_j'\epsilon^je^{j\lambda_r't}\right)^{2}}+r(K\star_{\tilde  c} \psi_\sigma)(t)\sum_{j=1}^ka_j'\epsilon^je^{j\lambda_r't}
$$
$$
+\phi_\sigma(t)\sum_{j=1}^k(rb'_j\kappa_j(r, \tilde c) -2a_j') \epsilon^je^{j\lambda_r't} + r\sum_{j=1}^ka_j'\epsilon^je^{j\lambda_r't}\sum_{j=1}^kb_j'\kappa_j(r, \tilde c)\epsilon^je^{j\lambda_r't}=\phi_\sigma(t)\Delta(t,\epsilon),  
$$
where 
$$
\Delta(t,\epsilon)= -(\tilde c-c)\frac{\phi'(t)}{\phi(t)} -(\sigma-1)(\phi(t)- r(K\star_{\tilde c} \psi)(t))-r{\gamma(t)} + \epsilon e^{\lambda_r't}L_{k-1}(\epsilon e^{\lambda_r't})+
$$
$$
r\epsilon e^{\lambda_r't}\frac{(K\star_{\tilde c} \psi)(t)}{\phi(t)}F_{k-1}(\epsilon e^{\lambda_r't})+\epsilon^{k+1}\frac{e^{(k+1)\lambda_r't}}{\phi(t)}M_{k-1}(\epsilon e^{\lambda_r't}),
$$
 $L_{k-1}, M_{k-1}, N_{k-1}$ are some real polynomials of degree $k-1$.  
Now, since $$\psi(t-\tilde cs-y)/\phi(t) \leq \sup_{t \in \R}\psi(t+m)/\phi(t) \in \R$$ for $s \geq 0, y \geq -m$, an application of
the  Lebesgue's bounded convergence theorem   together with the Proposition 3 
yields  that 
$$\lim_{t \to -\infty}\frac{(K\star_{\tilde c} \psi)(t)}{\phi(t)}= \frac{b}{1-r}\int_0^{m}\int_{-m}^m K(s,y)e^{-\omega(\tilde cs+y)}dy\; ds , \ \omega \in \{\lambda_r(c), \mu_r(c)\}. 
$$ 

 Taking into account  the above relation together with (\ref{3w}) and the relation $e^{(k+1)\lambda_r't} = o(\phi_\sigma(t))$, $t \to -\infty$, 
 we can conclude that ${\cal D}_1(\Phi_+,\Psi_+)(t) <0,$ $t \in (-\infty, -3-m(1+c)]$,  for all small $\epsilon >0$. Arguing as in Section \ref{Sec4}, we finally obtain that ${\cal D}_1(\Phi_+,\Psi_+)(t)  <0, \ t \in \R$ for all small positive $\epsilon$. 

The proof of the inequality ${\cal D}_2(\Phi_+,\Psi_+)(t) <0, \ t \in \R,$ is the same as in (\ref{22}) (where $c$ should be replaced with $\tilde c$). Next, we can take   $\Psi_-(t)\equiv 0$ and $\Phi_-(t)$ defined in  Section \ref{Exi} as suitable lower solutions (again, in the definition of $\Phi_-(t)$,  $c$ should be replaced with $\tilde c$).  As we know from Section \ref{Exi}, this implies the existence of monotone wavefront $(\Phi(t), \Psi(t))$ propagating with the speed  $\tilde c$. Moreover, arguing as in the second part of  Section \ref{Exi}, we see that this existence result is also valid in the case of normalised  kernels $K$ with non-compact supports. 
 
\noindent (c) Fix now $b >0, \ r \in (0,1)$ and consider 
$$
\frak{C}(b,r) := \{ c \geq 0: \ {\rm system} \ (\ref{3ade}) \ {\rm has\ a\ wavefront \ propagating \ at \ the\ velocity\ } c \}.
$$
As we already know, $\frak{C}(b,r)$ contains the  interval $[2, +\infty)$. Assume now that 
$c \in \frak{C}(b,r), \ c \geq 2\sqrt{1-r},$ and take an arbitrary $\tilde c> c$. Then  there exists a monotone
traveling front for (\ref{3ade})   propagating at the velocity $\tilde c$.  As a consequence, for each $c \in \frak{C}(b,r)$ we obtain  $[c, +\infty) \subset \frak{C}(b,r)$ so that $\frak{C}(b,r)$ is a proper connected unbounded subinterval of $[2\sqrt{1-r}, +\infty)$.  Set $c_*(b,r):= \inf  \frak{C}(b,r)$, then
$c_*(b,r)\in  \frak{C}(b,r)$ by a standard limiting argument applied to the sequence of wavefronts $(\Phi_j(x+c_jt), \Psi_j(x+c_jt))$,  $c_j =c_*(b,r)+1/j$, normalised by the condition $\Phi_j(0) =0.5$ (see Step 2 of Lemma \ref{Le2} for more details).   By Proposition \ref{mudas}, $c_*(b,r) \geq 2\sqrt{1-r},$ which completes the proof of  the theorem. 
\qed

%%%%%%%%%%%%%%%%%%%%%%%%%%%%%%%%%%%%%%%%%%%%%%%%%%%%%%%%%%%%%%%%%%%%%%%%
%%%%%%%%%%%%%%%%%%%%%%%%%%%%%%%%%%%%%%%%%%%%%%%%%%%%%%%%%%%%%%%%%%%%%%%%
\section{Pushed minimal wavefronts in the BZ system}\label{Sec5}
%%%%%%%%%%%%%%%%%%%%%%%%%%%%%%%%%%%%%%%%%%%%%%%%%%%%%%%%%%%%
Our first result shows that large $b>0$ is one of the reasons for the appearance of nonlinearly determined minimal wavefronts in system  (\ref{3ade}). 
\begin{lem}  \label{dac} Fix  $r\in (0,1)$ and suppose that the kernel $K(s,y)$ is a continuous function of $s, y$. Then $c_*(r,b,K) \to 2$ as $b \to +\infty$. In particular, the critical 
wavefront  is obligatorily  pushed  once  $b$ is a sufficiently large positive number. 
\end{lem} 
While proving this assertion, we will use  estimate (\ref{oza}) given below:
\begin{lem}  \label{mdaF}  Suppose that $b \geq 1,  \ r \in (0,1)$.    Then
\begin{equation}\label{oza}
\phi(t) < \psi(t) < \frac{b}{1-r}\phi(t), \quad t \in \mathbb{R}.
\end{equation}
\end{lem} 
\begin{proof} 
Set $z(t):=b_1\phi(t)- \psi(t)$, where $b_1= b/(1-r)$. 
It is easy to see that
$$
z''(t)-cz'(t)
+\phi(t)\left\{b_1(r(K\star \psi)(t)-\phi(t)) + b\psi(t)\right\}=0.
$$
Since $z(-\infty)=0,\  z(+\infty)=b_1 -1>0,$ we claim that $z(t) >0$ for all $t\in \mathbb{R}.$ Arguing by contradiction, 
we assume  the existence of  $\tau \in \mathbb{R}, $
such that $z(\tau)\leq 0, \ z'(\tau)= 0, \\
z''(\tau)\geq 0$. But  $z(\tau)\leq 0$ implies
$
b_1\phi(\tau)\leq  \psi(\tau)
$
and therefore 
\begin{eqnarray*}
0\geq  b_1(r(K\star \psi)(\tau)-\phi(\tau)) + b\psi(\tau)\geq b_1r(K\star \psi)(\tau)-\psi(\tau) + b\psi(\tau)>0,
\end{eqnarray*}
 a contradiction proving the right-side inequality in (\ref{oza}). 
 
 Next, set $z(t):=\phi(t)- \psi(t)$.  
We have  $z(-\infty)= z(+\infty)= 0,$
so that the non-negativity of $z$ at some points would imply the existence of 
$\tau$ such that $z(\tau)\geq  0, \ z'(\tau)= 0,$ $z''(\tau)\leq 0$. But then $
\phi(\tau)\geq  \psi(\tau)
$ 
and therefore 
\begin{eqnarray*}
0= z''(\tau) + (1-r-b)\phi(\tau) 
+\phi(\tau)(r(K\star \psi)(\tau)-\phi(\tau))+ b \phi(\tau)\psi(\tau), 
\end{eqnarray*}
$$
0 \leq 1-r-b
+r(K\star \psi)(\tau)-\phi(\tau)+ b\psi(\tau) \leq 1-r-b
+r(K\star \psi)(\tau)-\psi(\tau)+ b\psi(\tau)=
$$
$$
-r(1-(K\star \psi)(\tau)) -(b-1)
(1-\psi(\tau))<0, 
$$
a contradiction proving the left-side inequality in (\ref{oza}). \qed
\end{proof}

\begin{proof}[Lemma \ref{dac}] Supposing that $\liminf_{b \to +\infty} c_*(r,b,K) =c_0 <2$, we can find  a sequence $b_j\to +\infty$ such that 
$\lim_{j \to +\infty} c_*(r,b_j,K) =c_0$.  Fix   $c_1\in (c_0,2)$, then the  system
\begin{equation}\label{3adej}\left\{
\begin{array}{ll}
      \phi''(t) - c_1\phi'(t) + \phi(t) (1 - r- \phi(t)+ r(K\star \psi)(t)) =0,
    &    \\
     \psi''(t) - c_1\psi'(t) +b_j \phi(t)(1-\psi(t)) =0, & \\      \phi>0, \psi <1, \      \phi(-\infty)=\psi(-\infty)=0, \  \phi(+\infty)=\psi(+\infty)=1,  &\\ 
\end{array}%
\right.
\end{equation}
 has a monotone wavefront $(\phi_j(t),\psi_j(t))$ normalised by the condition $\phi_j(0)=0.5$. Let $s$ be a critical point of the derivative $\phi'_j$.  Since $\phi''_j(s)=0$, we obtain from the first equation in (\ref{3adej}) that 
 $$\phi_j'(s) = c_1^{-1}\phi_j(s) (1 - r- \phi_j(s)+ r(K\star \psi_j)(s)) \leq (4c_1)^{-1}. 
 $$
 Thus $\sup_{t\in \R} |\phi_j'(t)| \leq 1/(4c_1)$ and by the Arzel\`a-Ascoli theorem, 
 without loss of generality, we can assume that, uniformly on the compact intervals, 
 $\phi_j(t)\to \phi_*(t),$ where $\phi_*: \R \to [0,1], \ \phi_*(0)=0.5,$ is a monotone  continuous function. By the Helly selection theorem,  we can further assume that  $\psi_j(t) \to \psi_*(t)$ point-wise on $\R$, 
where $\psi_*:  \R \to [0,1]$ is a monotone function such that $\psi_*(t)\geq \phi_*(t), \ t \in \R$, see (\ref{oza}). 
Next, as a bounded solution to the first equation in (\ref{3adej}), $\phi_j$ satisfies the  integral equation 
$$
\phi_j(t) = \frac{1}{\mu_r(c_1)-\lambda_r(c_1)}
\int_t^{+\infty}(e^{\lambda_r(c_1) (t-s)}- e^{\mu_r(c_1)
(t-s)})\phi_j(s)\left(\phi_j(s)- r(K\star \psi_j)(s)\right)ds.
$$
Taking the  limit for $j\to +\infty$ implies: 
$$
\phi_*(t) = \frac{1}{\mu_r(c_1)-\lambda_r(c_1)}
\int_t^{+\infty}(e^{\lambda_r(c_1) (t-s)}- e^{\mu_r(c_1)
(t-s)})\phi_*(s)\left(\phi_*(s)- r(K\star \psi_*)(s)\right)ds. 
$$ 
Thus $\phi_*(t)$ satisfies the following  %differential 
equation 
%with continuous coefficients 
\begin{equation}\label{ep}
\phi''_*(t) - c_1\phi'_*(t) + \phi_*(t) (1 - r- \phi_*(t)+ r(K\star \psi_*)(t)) =0.
\end{equation}
 Since $\phi_*(t) \geq 0$ and $\phi_*(0)=0.5,$ it follows that $\phi_*(t) >0$ for all $t \in \R$. 
%
% $$\phi_j(t),\psi_j(t)) \to (\phi_*(t),\psi_*(t)),$$
%where $(\phi_*(t),\psi_*(t))$ is a monotone wavefront to (\ref{3ade}) satisfying the boundary conditions 
 %$$\phi(-\infty)=0, \ \psi(-\infty) \in [0,1], \  \phi(+\infty)=\psi(+\infty)=1.  $$
 
Finally, after integrating the second equation in (\ref{3adej}) on $\R$ and applying the Fatou lemma, we find that 
$$
\int_\R\phi_*(t)(1-\psi_*(t))dt \leq \liminf_{j \to +\infty}\int_\R\phi_j(t)(1-\psi_j(t))dt \leq \lim_{j \to +\infty}\frac{c_1}{b_j} = 0.
$$
Since $0<\phi_*(t) \leq 1, \ 0\leq \psi_*(t) \leq 1, \  \ t \in \R$, this can happen only if $\psi_*(t)\equiv 1$. Consequently, 
we see from (\ref{ep})  that $\phi_*(t)$ is the standard KPP-Fisher profile: 
$$
 \phi_*''(t) - c_1\phi_*'(t) + \phi_*(t) (1 - \phi_*(t)) =0. 
$$
This, however,  implies that $c_1 \geq 2$, a contradiction. \qed
\end{proof} 
Proof of Lemma \ref{dac} also exhibits the limit form of the wavefront $(\phi_b(t),\psi_b(t)), \phi_b(0)=0.5,$ when $b \to +\infty$ and $c \geq 2, r \in (0,1)$ are fixed. Namely, $\lim_{b \to +\infty}(\phi_b(t),\psi_b(t)) = (\phi_*(t),1)$ uniformly on compact sets, where $\phi_*$ is defined above.  The convergence $\lim_{b \to +\infty}\psi_b(t) = 1$ can also be seen from the following proposition:
\begin{lem}  \label{mdaT}  Suppose that $b \geq 1,  \ r \in (0,1), $ $c > 2 \sqrt{1-r}$ and let $\rho(t)$ be the unique positive monotone front 
of the KPP-Fisher equation 
\begin{equation}\label{rp}
\rho''(t)-c\rho'(t)+ (1-r)\rho(t)(1-\rho(t))=0,  
\end{equation}
normalised by the condition 
$$
\lim_{t \to -\infty} \rho(t)e^{-\lambda_r(c) t}= \lim_{t \to -\infty} \psi(t)e^{-\lambda_r(c) t} = \frac{b}{1-r}.$$
Then 
\begin{equation}\label{ozazu}
\rho(t) < \psi(t), \quad t \in \mathbb{R}.
\end{equation}
\end{lem} 
\begin{proof} By Lemma \ref{mdaF}, 
$$
0= \psi''(t)-c\psi'(t) + b\phi(t)(1-\psi(t)) > \psi''(t)-c\psi'(t) + (1-r)\psi(t)(1-\psi(t)), \ t \in \R,
$$
so that $\psi(t)$ is an upper solution for (\ref{rp}).  Next, 
it is easy to verify that, for appropriate $L \gg 1$ and $0 < \delta \ll 1$, 
$$
\psi_-(t)=  \frac{b}{1-r}\max\{0, e^{\lambda_r(c) t}(1-Le^{\delta t})\}
$$
is a lower solution of (\ref{rp}) satisfying the inequality $\psi_-(t) < \psi(t), \ t \in \R$. As a consequence, there exists a  wavefront $\hat \psi(t)$ such that $\psi_-(t) < \hat \psi(t) < \psi(t)$ and 
$$
\lim_{t \to -\infty} \hat \psi(t)e^{-\lambda_r(c) t} = \frac{b}{1-r}.$$
Since the wavefront %$\rho(t)$ 
to  (\ref{rp}) is unique (up to a translation), we can conclude that  $\rho(t)=\hat \psi(t)$ and inequality (\ref{ozazu}) follows. \qed
\end{proof}

\begin{example}[Continuation of Example 7]\label{wsc} Here we present numerical simulations confirming the conclusion of Lemma \ref{dac}. For the weak delay kernel case,  system (\ref{1}) can be reformulated as  follows:
\begin{equation}\label{3equations}
\begin{array}{ll}
     u_t(t,x) = \Delta u(t,x)  + u(t,x)(1-u(t,x)-rw(t,x)), &    \\
     v_t(t,x) = \Delta v(t,x)  -b u(t,x)v(t,x), &  \\ 
     w_t(t,x) = \Delta w(t,x)  + \frac{1}{\tau} (v(t,x)-w(t,x)),
\end{array}%
\end{equation}
where $w(t,x)=(K_w*v)(t,x)$, see \cite{DQ} for details.
We will consider the initial functions 
\begin{equation} \label{bce3}
u(0,x)= \begin{cases} 
      0 & x < 0, \\
      1 & x \geq 0,
      \end{cases}, \quad 
v(0,x)=w(0,x)= \begin{cases} 
      1 & x < 0, \\
      0 & x \geq 0.
      \end{cases}
\end{equation}
It is a "folk theorem", proved for many relevant monostable models, that  initial signals with bounded (or semi-bounded as in (\ref{bce3})) supports  should invade the unoccupied  side of space with the asymptotic speed of propagation $c_\circ$  equal to the minimal speed $c_*$. We will use this principle to estimate numerically the minimal speed $c_*(r,b,\tau)$ for system (\ref{1}) considered with the weak delay kernel. 
Our simulations are based on the Crank-Nicholson method which is second-order accurate in both spatial and temporal directions. 
The spatial step size is chosen as $\Delta x=0.05$ in the computational interval $x \in [-450,50]$ together with the Dirichlet boundary conditions $u(t,-450)=0$, $u(t,50)=1$ and $v(t,-450)=w(t,-450)=1$, $v(t,50)=w(t,50)=0$. The temporal step size is $\Delta t=0.01$.

The table on Figure 1 (left) gives numerically calculated speeds of propagation  for some values of  $b$, $\tau$ and $r=0.75$. These instantaneous speeds are computed at some appropriate moments $t_{b,r} \in [250,400]$ which depend on $b, r$. `Unsettled'  digits that would change if we will continue to calculate speeds for larger moments $t_{b,r}$  are marked in red. Interestingly, for the highest values of the parameter $b$, we obtain reliable estimates of minimal speeds at the lowest  values of $t_{b,r}$. We attribute this phenomenon to the better stability properties of pushed waves. Obviously, the obtained values cannot be understood as the precise asymptotic speeds of propagation $c_\circ(3/4,b, \tau)$.

Particularly, for each $\tau >0$ the linearly determined minimal speed $c_*(3/4,b,\tau)$ is 1. Actually, it can be seen from (\ref{weakkernel})  that $c_*(3/4,b,\tau)=1$ for $b \in (0,1/3+\tau/12]$. It is worth noting rather slow, `logarithmic scale' convergence, as $b \to +\infty$, of  the  speed of propagation $c_\circ(3/4,b,\tau)\approx c_*(3/4,b,\tau)$ to its  limit value 2.  See also the right frame of Figure 1  for the dynamics of component $u$ of the solution for problem (\ref{3equations}), (\ref{bce3}).  
\begin{figure}[h]\label{Fig1}
\vspace{-33mm}

\hspace{-20mm} {\includegraphics[width=9cm]{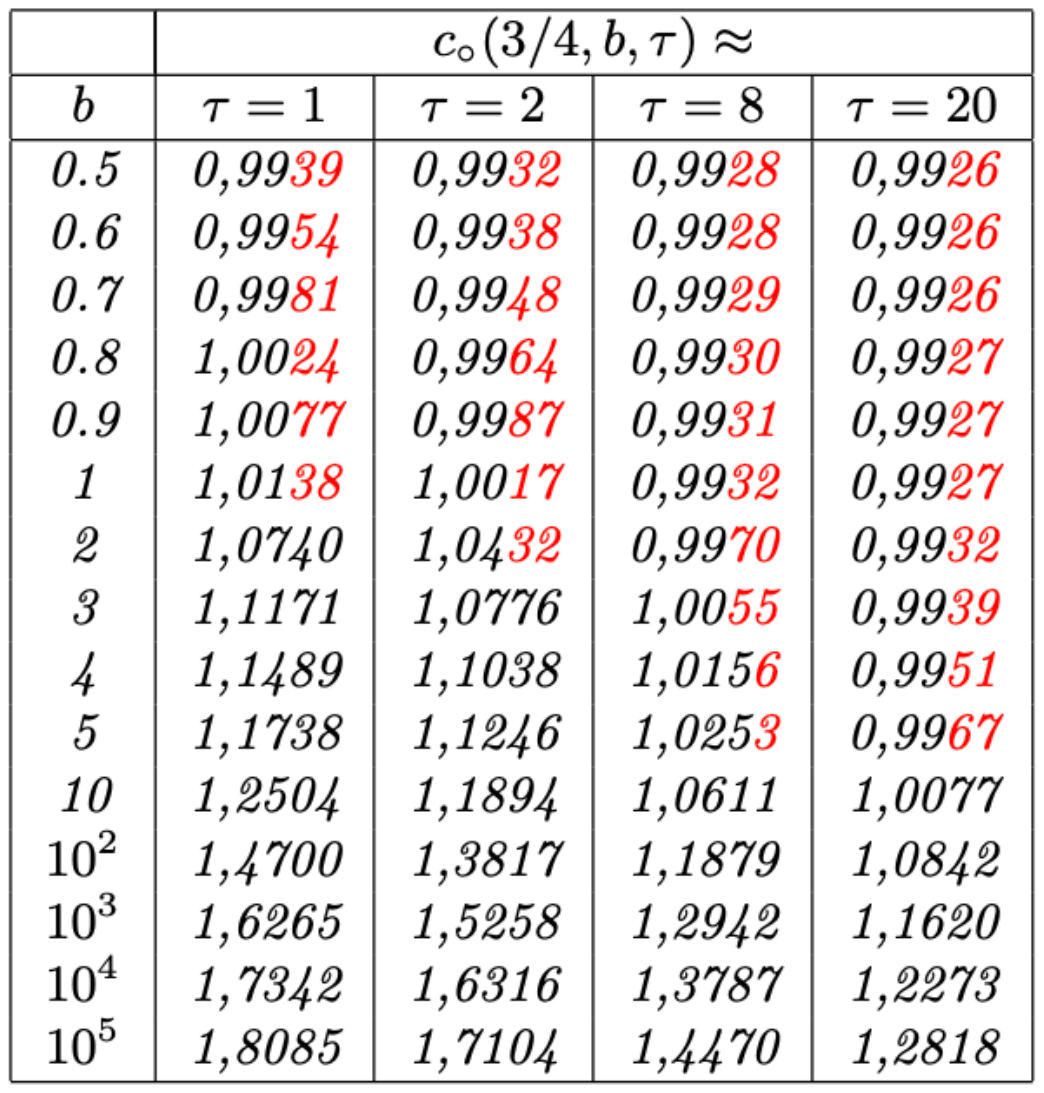}}\hspace{-15mm}{\includegraphics[width=9cm]{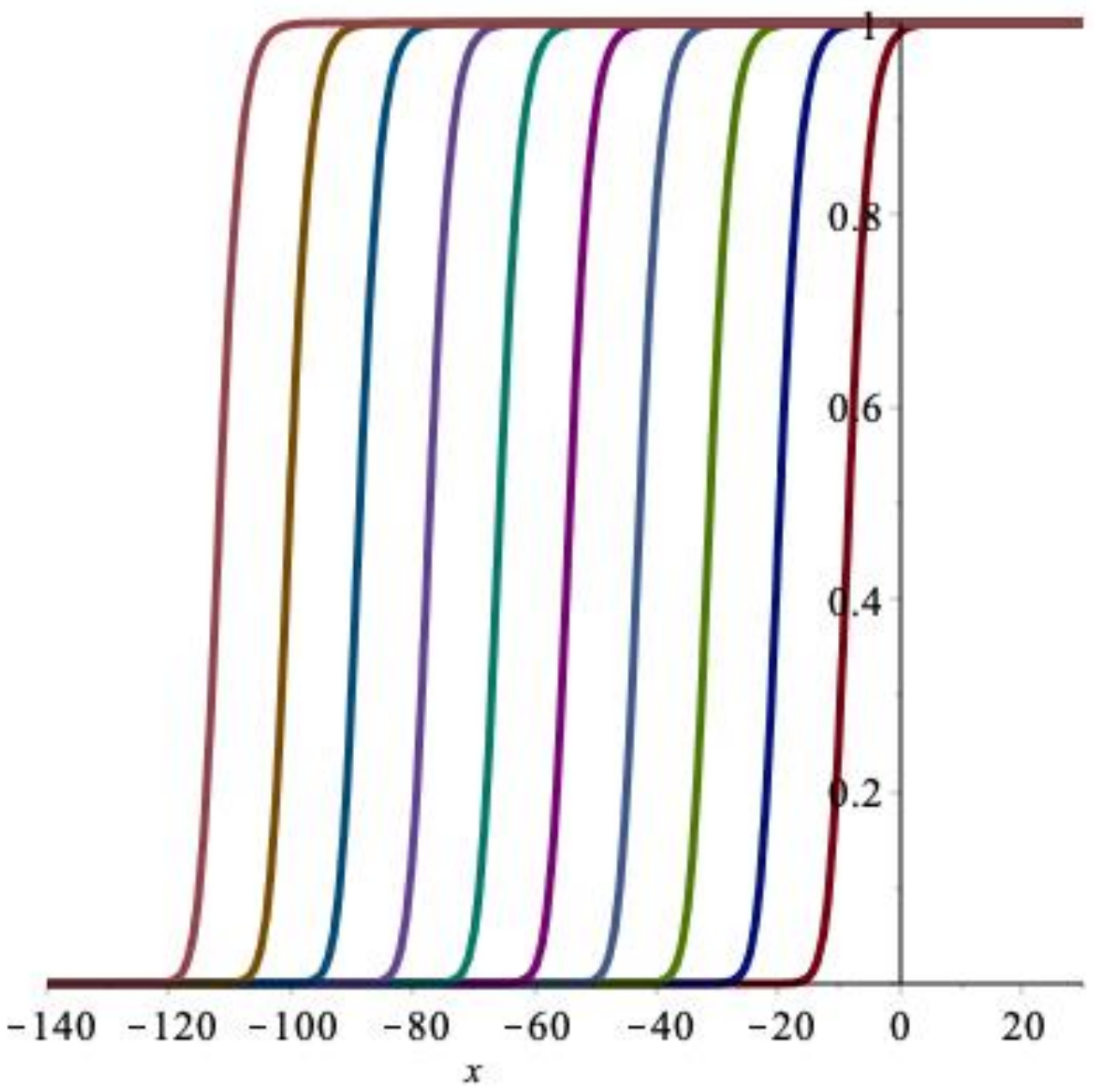}}

\vspace{-33mm}
\caption{On the left: numerically calculated speeds of propagation  for some values of  $b$, $\tau$ and $r=0.75$. On the right: graph of the component $u$ of the solution for system  (\ref{3equations}) considered with $\tau = 1$, $r=0.75$, $b=4$.  The solution is shown in times $t=10, 20, 30, ..., 100$.
}
\end{figure}
\end{example}

As our next result shows, the second reason for the appearance of pushed wavefronts  lies in the  asymmetry of the influence of the density distribution $u(t,x)$ on the  dynamics: 
\begin{lem}  \label{dac1} Fix some $r\in (0,1)$ and  set $K_a(s,y):=K(s,y+a)$.  Then $c_*(a):=c_*(r,b,K_a) \to 2$ as $a \to +\infty$. \end{lem} 

\begin{proof} If $\liminf_{a \to +\infty} c_*(a) =c_0 <2$, then there exists a sequence $a_j\to +\infty$ such that 
$\lim_{j \to +\infty} c_*(a_j) =c_0$.  Fix some  $c_1\in (c_0,2)$, then the  system
\begin{equation}\label{3adea}\left\{
\begin{array}{ll}
      \phi''(t) - c_1\phi'(t) + \phi(t) (1 - r- \phi(t)+ r(K_{a_j}\star \psi)(t)) =0,
    &    \\
     \psi''(t) - c_1\psi'(t) +b\phi(t)(1-\psi(t)) =0, & \\      \phi>0, \psi <1, \      \phi(-\infty)=\psi(-\infty)=0, \  \phi(+\infty)=\psi(+\infty)=1,  &\\ 
\end{array}%
\right.
\end{equation}
has a monotone wavefront $(\phi_j(t),\psi_j(t))$ normalised by the condition $\phi_j(0)=0.5$. 
It follows from the proof of Lemma \ref{dac} that  $\sup_{t\in \R} |\phi_j'(t)| \leq 1/(4c_1)$ and similar arguments show that 
$\sup_{t\in \R} |\psi_j'(t)| \leq b/c_1$.  Thus, by the Arzel\`a-Ascoli theorem, 
 without loss of generality, we can assume that, uniformly on the compact intervals, 
$$(\phi_j(t),\psi_j(t)) \to (\phi_*(t),\psi_*(t)),$$
where $0\leq  \phi_*(t), \psi_*(t)\leq 1, \ t \in \R, \ \phi_*(0)=0.5$, are continuous monotone functions. 
After integrating the second equation in (\ref{3adea}) on $\R$ and applying the Fatou lemma, we find that 
$$
\int_\R\phi_*(t)(1-\psi_*(t))dt \leq \liminf_{j \to +\infty}\int_\R\phi_j(t)(1-\psi_j(t))dt = \frac{c_1}{b}.
$$
Since $0.5\leq \phi_*(t) \leq 1, \ t \geq 0 $, this can happen only if $\psi_*(+\infty)= 1$.

Furthermore, applying the Helly selection theorem,  we can assume that  $$(K_{a_j}\star\psi_j)(t)  \to p_*(t)$$ point-wise on $\R$, 
where $p_*:  \R \to [0,1]$ is a monotone function.  
Thus for every fixed positive $\alpha$ and sufficiently large $j$, the monotonicity of $\psi_j$ implies that 
$$
\lim_{j \to +\infty} (K_{a_j}\star\psi_j)(t)= \lim_{j \to +\infty} \int_0^{+\infty}\int_\R K(s,y)\psi_j(t-cs-y+a_j)dy\; ds\geq
$$
$$
\lim_{j \to +\infty} \int_0^{+\infty}\int_\R K(s,y)\psi_j(t-cs-y+\alpha)dy\; ds= \int_0^{+\infty}\int_\R K(s,y)\psi_*(t-cs-y+\alpha)dy\; ds. 
$$
Thus, for every fixed $t \in \R$, 
$$
\lim_{j \to +\infty} (K_{a_j}\star\psi_j)(t)\geq \sup_{\alpha>0}  \int_0^{+\infty}\int_\R K(s,y)\psi_*(t-cs-y+\alpha)dy\; ds=\psi_*(+\infty)=1.$$
Therefore, taking into account that $(K_{a_j}\star\psi_j)(t)\leq 1$, we conclude that 
$$
\lim_{j \to +\infty} (K_{a_j}\star\psi_j)(t)=1, \ t \in \R.
$$
As a bounded solution of the first equation in (\ref{3adea}), $\phi_j$ satisfies the following integral equation 
$$
\phi_j(t) = \frac{1}{\mu_r(c_1)-\lambda_r(c_1)}
\int_t^{+\infty}(e^{\lambda_r(c_1) (t-s)}- e^{\mu_r(c_1)
(t-s)})\phi_j(s)\left(\phi_j(s)- r(K_{a_j}\star \psi_j)(s)\right)ds,
$$
from where, after taking the limit as $j\to +\infty$,  
$$
\phi_*(t) = \frac{1}{\mu_r(c_1)-\lambda_r(c_1)}
\int_t^{+\infty}(e^{\lambda_r(c_1) (t-s)}- e^{\mu_r(c_1)
(t-s)})\phi_*(s)\left(\phi_*(s)- r\right)ds. 
$$ 
Thus the monotone function $\phi_*(t)$ satisfies the KPP-Fisher differential equation 
$$
\phi''(t) - c_1\phi'(t) + \phi(t) (1 - \phi(t)) =0.
$$
 Since  $0\leq \phi_*(t) \leq 1$ and $\phi_*(0)=0.5$, we have  that $0<\phi_*(t) <1$ for all $t \in \R$. 
In addition, $\phi_*: \R \to (0,1)$ is a non-decreasing  function. This means that $\phi_*(t)$ is a standard KPP-Fisher profile and, consequently,
 $c_1 \geq 2$, a contradiction. \qed
\end{proof} 
\section*{Appendix}
\begin{lem} \label{OWL} Fix $c>0$ and suppose that $b>0$ and $\phi(t)$ is a positive non-decreasing continuous  function such that $\phi(+\infty)=1$ and 
$\int_{-\infty}^0\phi(t)\;dt$ is finite.  Then there exists a unique strictly monotone $C^2$-smooth solution $\psi(t)$ of the boundary value problem 
$$
\psi''(t)-c\psi'(t) + b\phi(t)(1-\psi(t))=0, \quad \psi(-\infty)=0, \ \psi(+\infty)=1.  
$$
This defines the operator $\frak{L}_b$, $\psi= \frak{L}_b\phi$ on the respective functional sets.  $\frak{L}_b$ commutes with the translation operator, $(\frak{L}_b \phi(\cdot+h))(t)= (\frak{L}_b \phi(\cdot))(t+h)$, and is 
monotone increasing with respect to $b$ and $\phi$: 
\begin{itemize}
\item[a)] if $\phi_1(t) \leq \phi_2(t)$ then  
$(\frak{L}_b\phi_1)(t) \leq  (\frak{L}_b\phi_2)(t), \ t \in \R$;
\item[b)]  if $b_1 < b_2$ then $(\frak{L}_{b_1}\phi)(t) < (\frak{L}_{b_2}\phi)(t), \ t \in \R$, for each $\phi$ from the domain of $\frak{L}$. 
\end{itemize} 
\end{lem}
\begin{proof}  The change of  
variables $\psi(t)=1-\theta(t)$, 
transforms the above problem into 
\begin{equation}\label{3adekul}
     \theta''(t) - c\theta'(t) -b \phi(t)\theta(t) =0,  \quad    \theta(-\infty)=1, \  \theta(+\infty)=0.  
\end{equation}
Since the limit equation of (\ref{3adekul}) at $+\infty$  is $\theta''(t) - c\theta'(t) -b\theta(t) =0$, by the Hartman asymptotic theory \cite[Theorem 17.4, p. 317]{H},  (\ref{3adekul}) has a fundamental system of solutions $\theta_1(t), \theta_2(t)$ such that  
$\lim_{t\to +\infty}\theta_i'(t)/\theta_i(t)= \mu_i$, $i=1,2,$ where $\mu_2<0<\mu_1$ are the roots of the characteristic equation 
$z^2-cz-b=0$.  {Clearly,  we choose $\theta_2$ such that $\theta_2(t) >0$ for all large positive $t$.  Consequently, each non-negative nontrivial solution to (\ref{3adekul}) vanishing at $+\infty$ has the form $\theta(t) = C\theta_2(t)$ for some $C>0$.  We claim that $\theta_2'(t)<0$ for all $t \in \R$. 
This property is obvious for large positive $t$.  Arguing by contradiction, let $s<0$ denotes the rightmost point 
where $\theta_2'(s) =0,$ then $\theta_2''(s) \leq 0$ and $\theta_2(s) >0$, all these relations are incompatible with (\ref{3adekul}) at the point $s$. 

 Next,  equation (\ref{3adekul}) is asymptotically autonomous at $-\infty$ and $\phi \in L_1(\R_-)$. Therefore, by the asymptotic integration theory (see e.g. Theorem 1.8.1 in \cite{MSPE}), since $0 < c$ are the only 
eigenvalues of the limiting equation $   \theta''(t) - c\theta'(t)=0$ at $-\infty$,  there exists the  positive finite limit $\lim_{t \to -\infty}\theta_2(t)$.   Clearly,  the unique 
solution satisfying also the boundary conditions in (\ref{3adekul}) is $\theta(t)= \theta_2(t)/\theta_2(-\infty).$ We define  $\frak{L}_b\phi=1-\theta.$ }

Now, assume $\phi_1(t) \leq \phi_2(t)$ and consider the corresponding solutions $\theta_1, \theta_2$ to  the problem (\ref{3adekul}). Then $\delta(t)= \theta_2(t)-\theta_1(t)$ satisfies the boundary value problem 
$$
  \delta''(t) - c\delta'(t) -b(\phi_2(t)\theta_2(t)-\phi_1(t)\theta_1(t)) =0, \ \delta(-\infty)= \delta(+\infty)=0.
$$
If $\delta(p) >0$ at some point $p,$ we can find a critical point $s$ such that 
$\delta(s) > 0,$ $\delta'(s)=0,$ $\delta''(s)\leq 0$. Since 
$$\phi_2(s)\theta_2(s)-\phi_1(s)\theta_1(s)=\phi_2(s)\delta(s)+(\phi_2(s)-\phi_1(s))\theta_1(s)>0,$$
we arrive to a contradiction with the above differential equation for $\delta(t)$ at $t=s$. Thus $\theta_2(t) \leq \theta_1(t), \ t \in \R$. 

Similarly, consider  $b_1<b_2$ and set  $\psi_j=\frak{L}_{b_j}\phi$.  Then $\sigma(t)= \psi_1(t)-\psi_2(t)$ satisfies the boundary value problem 
$$
  \sigma''(t) - c\sigma'(t) -\phi(t)(b_2\theta_2(t)-b_1\theta_1(t)) =0, \ \sigma(-\infty)= \sigma(+\infty)=0.
$$
If $\sigma(p) \geq 0$ at some point $p$ we can find a critical point $s$ such that 
$\sigma(s) \geq 0,$ $\sigma'(s)=0,$ $\sigma''(s)\leq 0$. Since 
$$b_2\theta_2(s)-b_1\theta_1(s)=b_2\sigma(s)+(b_2-b_1)\theta_1(s)>0,$$
we again obtain a contradiction. Thus $\psi_1(t) < \psi_2(t), \ t \in \R$. \qed 
\end{proof} 
Under appropriate conditions, linear operators commuting with the translations  have good continuous properties and can be represented as convolutions, see \cite{meis} and references therein. $\frak{L}_b$ is not defined on a linear space. Nevertheless, its monotonicity immediately implies a sort of  continuity.  We will say that the function $\zeta$ belongs to an $\epsilon-$neighbourhood of $\phi$ ($\zeta$, $\phi$ from the class defined in Lemma \ref{OWL}) if 
$\phi(t-\epsilon) \leq \zeta(t) \leq \phi(t+\epsilon)$, $t \in \R$.  This concept defines the convergence $\zeta_j \stackrel{\epsilon}\to \phi$ in a natural way: there exists a sequence $\epsilon_j \to 0^+$ such that 
$\phi(t-\epsilon_j) \leq \zeta_j(t) \leq \phi(t+\epsilon_j)$, $t \in \R$. Then, by the monotonicity property, 
$(\frak{L}_b\phi)(t-\epsilon_j) \leq (\frak{L}_b)\zeta_j(t) \leq (\frak{L}_b\phi)(t+\epsilon_j)$, so that $\frak{L}_b\zeta_j \stackrel{\epsilon}\to \frak{L}_b \phi$.  

It is also convenient to extend the domain of $\frak{L}_b$ to the steady states $\phi =0$ and $\phi=1$: for this we interpret $0, 1$ 
as the limit values $1=\lim_{j \to +\infty}\phi_j(t), 0 =\lim_{j \to -\infty}\phi_j(t)$, where $\phi_j(t)= \phi(t+j)$ and the limits  hold uniformly on each half-line $[q, +\infty)$. Then clearly $\lim_{j \to +\infty}\frak{L}_b\phi_j(t)=\lim_{j \to +\infty}\psi(j+t)=1$, $\lim_{j \to -\infty}\frak{L}_b\phi_j(t)=\lim_{j \to -\infty}\psi(j+t)=0$ on the same sets.  Consequently, we define  $\frak{L}_b 0=0, \frak{L}_b1 =1$. 

\section*{Acknowledgments}  \noindent  The work of Karel Has\'ik, Jana Kopfov\'a and Petra N\'ab\v{e}lkov\'a was supported  by the institutional support
for the development of research organizations I\v CO 47813059.
%This work was realized during a stay of Sergei Trofimchuk at the Silesian University in Opava,  Czech Republic. This stay was possible due to the support of  the Silesian University in Opava and the European Union through the project  CZ.02.2.69/0.0/0.0/16\_027/0008521.  
O. Trofymchuk and S. Trofimchuk  were  also supported by FONDECYT (Chile),   project 1190712. 
%\newpage 


\vspace{0mm}

\begin{thebibliography}{99}


\bibitem{bn} {  A.~Boumenir, V.~M. Nguyen,}   {Perron theorem in the monotone
iteration method for traveling waves in delayed reaction-diffusion}
equations, {J. Differential Equations},  244 (2008),  1551--1570.

\bibitem{DQ} Z. Du, Q. Qiao, The dynamics of traveling waves for a nonlinear Belousov-Zhabotinskii system, J. Differential Equations, 269 (2020),  7214--7230.


\bibitem{MSPE} { M.S.P. Eastham,} { The asymptotic solution of  linear differential systems, } London Math.  Soc. Monogr.  Ser.,  Clarendon Press,  Oxford, 1989. 

\bibitem{FZ} { J.~Fang, X.-Q.~Zhao, }  {Monotone wavefronts of the nonlocal Fisher-KPP 
equation},  Nonlinearity 24 (2011)  3043--3054 . 


\bibitem{KH} K.P. Hadeler,  {\it Topics in mathematical biology,} Lecture Notes on Mathematical Modelling in the Life Sciences, Springer, 2017. 


\bibitem{HR} Hadeler, KP., Rothe, F.:
{Travelling fronts in nonlinear diffusion equations}.
 \emph{J. Math. Biol.}     \textbf{2}, 251--263 (1975)

\bibitem{H}  P. Hartman, Ordinary Differential Equations 2nd Edition, Birkhauser, 1982

\bibitem{KKNT-20} K. Hasik,  J. Kopfov\'a,  P. N\'ab\v elkov\'a and  S. Trofimchuk, {On the geometric diversity of wavefronts for the scalar Kolmogorov ecological equation}, 
J. Nonlinear Science, 30 (2020), 2989--3026. 


\bibitem{MFa}   K. Hasik,  J.  Kopfov\'a, P.  N\'ab\v elkov\'a, S. Trofimchuk,  On pushed wavefronts of monostable equation with unimodal delayed reaction, Discrete Continuous Dynamical Systems,41 (2021) 5979--6000.  

\bibitem{MF}   K. Hasik,  J.  Kopfov\'a, P.  N\'ab\v elkov\'a, S. Trofimchuk, Traveling waves in the nonlocal KPP-Fisher equation: different roles of the right and the left interactions,  J. Differential Equations, 269 (2020),  7214--7230.




%
%\bibitem{jhz} {  J.~Huang, X.~Zou},   Existence of traveling wavefronts of delayed reaction diffusion systems without monotonicity,  Discrete Contin. Dyn. Syst., 9 (2003),  925--936.


\bibitem{wh1}  { W.~Huang, M. Han},  Non-linear determinacy of minimum wave speed for 
a Lotka-Volterra competition model,   {J. Differential Equations}, 251 (2011),  
1549--1561. 


\bibitem{Ka2} { Ya.~I. Kanel,}  {Existence of a traveling-wave type solutions for the
Belousov-Zhabotinskii system of equations II},  Sib. Math. J.,  32
(1991),   390--400.

\bibitem{LZii} 
\newblock X. Liang, X.-Q. Zhao, 
\newblock {Spreading speeds and traveling waves for abstract 
monostable evolution systems,} 
\newblock J. Functional Anal., {259} (2010), 857--903. 


%\bibitem{LL} {  G.~ Lin, W.-T. Li,}  {Travelling wavefronts of
%Belousov-Zhabotinskii system with diffusion and delay,}  Appl. 
%Math. Letters, 22 (2009),   341--346. 

%\bibitem{Lv}{ G.~Lv, M.~Wang,}  {Traveling wave front in diffusive and competitive Lotka-Volterra systems,}  Nonlinear Analysis, RWA, 11 (2010),  1323-1329. 


\bibitem{ma} { S.~Ma,}  {Traveling wavefronts for delayed reaction-diffusion systems via a
fixed point theorem,} {J. Differential Equations},  171 (2001), 
 294--314.
 
 \bibitem {MP} {J.~Mallet-Paret},   {The Fredholm alternative for
functional differential equations of mixed type,}  J. Dynam.
Differential Equations,   11 (1999), 1--48. 


\bibitem{meis}  G. H.  Meisters,  Linear operators commuting with translations on $\mathcal{D}(\mathbb{R})$ are continuous, Proceedings of the American Mathematical Society
Vol. 106, No. 4 (Aug., 1989), pp. 1079-1083. 

\bibitem{Mur1} { J.~D. Murray,}   {On traveling wave solutions in a model for
Belousov-Zhabotinskii reaction}, J. Theor. Biol.,  56 (1976),
 329--353.

\bibitem{Mur2} { J.~D. Murray,}  {Lectures on nonlinear differential equations. Models in biology,}  Clarendon Press, Oxford, 1977. 


\bibitem{TPTa}   E. Trofimchuk,  M. Pinto, S. Trofimchuk,  Traveling waves for a model of the Belousov-Zhabotinsky reaction, J. Differential Equations 254 (2013), 3690--3714.


%\bibitem{TPTb}  E. Trofimchuk,  M. Pinto, S. Trofimchuk,   2014 On the minimal speed of front propagation in a model of the Belousov- Zhabotinsky reaction,  {\it Discrete Contin. Dyn. Syst. - B} {\bf 19} 1769--1781

\bibitem{Troy} { W.~C. Troy,}  { The existence of traveling wave front solutions of a
model of the Belousov-Zhabotinskii reaction,} {  J. Differential Equations }, 36 (1980), 89-98.

\bibitem{Volp}{ A.~I. Volpert, V.~A. Volpert, and V.~A. Volpert,}  {Traveling Wave Solutions of Parabolic Systems,} Translations of Mathematical Monographs, Vol. 140, Amer. Math. Soc., Providence, 1994.



\bibitem{wz} {  J.~Wu, X.~Zou},  {Traveling wave fronts of
reaction-diffusion systems with delay,} {J. Dynam. Diff. Eqns.},  13
(2001),   651--687. 

\bibitem{yw} { Q.~Ye,  M.~Wang}, {Traveling wave front solutions of Noyes-Field
System for Belousov-Zhabotinskii reaction}, Nonlin. Anal. TMA, 11
(1987), ~1289--1302.

\bibitem{Zh} G.-B.  Zhang, Asymptotics and uniqueness of traveling wavefronts for a delayed model of the Belousov-Zhabotinsky reaction. Appl. Anal. 99 (2020), no. 10, 1639--1660.



\end{thebibliography}
\end{document}